\documentclass{tac}
\usepackage{amsmath,amssymb,tikz,rotating,xspace,hyperref}

\numberwithin{equation}{section}
\theoremstyle{plain}
\newtheorem{theorem}[equation]{Theorem}
\newtheorem{lemma}[equation]{Lemma}
\newtheorem{proposition}[equation]{Proposition}
\newtheorem{corollary}[equation]{Corollary}
\theoremstyle{definition}
\newtheorem{definition}[equation]{Definition}
\newtheorem{example}[equation]{Example}
\newtheorem{remark}[equation]{Remark}

\usetikzlibrary{matrix} 
\tikzstyle{dot}=[circle, draw=black, fill=black!25, inner sep=.4ex]
\newif\ifvflip\pgfkeys{/tikz/vflip/.is if=vflip}
\newif\ifhflip\pgfkeys{/tikz/hflip/.is if=hflip}
\newif\ifhvflip\pgfkeys{/tikz/hvflip/.is if=hvflip}
\newlength\morphismheight
\newlength\wedgewidth
\tikzset{width/.initial=1mm}
\makeatletter
\tikzstyle{morphism}=[font=\small,morphismshape]
\pgfdeclareshape{morphismshape}
{
    \savedanchor\centerpoint
    {
        \pgf@x=0pt
        \pgf@y=0pt
    }
    \anchor{center}{\centerpoint}
    \anchorborder{\centerpoint}
    \saveddimen\overallwidth
    {
        \pgfkeysgetvalue{/tikz/width}{\minwidth}
        \pgf@x=\wd\pgfnodeparttextbox
        \ifdim\pgf@x<\minwidth
            \pgf@x=\minwidth
        \fi
    }
    \savedanchor{\upperrightcorner}
    {
        \pgf@y=.5\ht\pgfnodeparttextbox
        \advance\pgf@y by -.5\dp\pgfnodeparttextbox
        \pgf@x=.5\wd\pgfnodeparttextbox
    }
    \anchor{north}
    {
        \pgf@x=0pt
        \pgf@y=0.5\morphismheight
    }
    \anchor{north east}
    {
        \pgf@x=\overallwidth
        \multiply \pgf@x by 2
        \divide \pgf@x by 5
        \pgf@y=0.5\morphismheight
    }
    \anchor{east}
    {
        \pgf@x=\overallwidth
        \divide \pgf@x by 2
        \advance \pgf@x by 5pt
        \pgf@y=0pt
    }
    \anchor{west}
    {
        \pgf@x=-\overallwidth
        \divide \pgf@x by 2
        \advance \pgf@x by -5pt
        \pgf@y=0pt
    }
    \anchor{north west}
    {
        \pgf@x=-\overallwidth
        \multiply \pgf@x by 2
        \divide \pgf@x by 5
        \pgf@y=0.5\morphismheight
    }
    \anchor{south east}
    {
        \pgf@x=\overallwidth
        \multiply \pgf@x by 2
        \divide \pgf@x by 5
        \pgf@y=-0.5\morphismheight
    }
    \anchor{south west}
    {
        \pgf@x=-\overallwidth
        \multiply \pgf@x by 2
        \divide \pgf@x by 5
        \pgf@y=-0.5\morphismheight
    }
    \anchor{south}
    {
        \pgf@x=0pt
        \pgf@y=-0.5\morphismheight
    }
    \anchor{text}
    {
        \upperrightcorner
        \pgf@x=-\pgf@x
        \pgf@y=-\pgf@y
    }
    \backgroundpath
    {
    \begin{scope}
        \pgfkeysgetvalue{/tikz/fill}{\morphismfill}
        \pgfsetstrokecolor{black}
        \pgfsetlinewidth{.7pt}
        \begin{scope}
        \pgfsetstrokecolor{black}
        \pgfsetfillcolor{white}
        \pgfsetlinewidth{.7pt}
                \ifhflip
                    \pgftransformyscale{-1}
                \fi
                \ifvflip
                    \pgftransformxscale{-1}
                \fi
                \ifhvflip
                    \pgftransformxscale{-1}
                    \pgftransformyscale{-1}
                \fi
                \pgfpathmoveto{\pgfpoint
                    {-0.5*\overallwidth-5pt}
                    {0.5*\morphismheight}}
                \pgfpathlineto{\pgfpoint
                    {0.5*\overallwidth+5pt}
                    {0.5*\morphismheight}}
                \pgfpathlineto{\pgfpoint
                    {0.5*\overallwidth + \wedgewidth}
                    {-0.5*\morphismheight}}
                \pgfpathlineto{\pgfpoint
                    {-0.5*\overallwidth-5pt}
                    {-0.5*\morphismheight}}
                \pgfpathclose
                \pgfusepath{fill,stroke}
        \end{scope}
    \end{scope}
    }
}
\makeatother
\newcommand{\tinymult}[1][dot]{
\smash{\raisebox{-2pt}{\hspace{-5pt}\ensuremath{\begin{pic}[scale=0.4,yscale=-1]
    \node (0) at (0,0) {};
    \node[#1, inner sep=1.5pt] (1) at (0,0.55) {};
    \node (2) at (-0.5,1) {};
    \node (3) at (0.5,1) {};
    \draw (0.center) to (1.center);
    \draw (1.center) to [out=left, in=down, out looseness=1.5] (2.center);
    \draw (1.center) to [out=right, in=down, out looseness=1.5] (3.center);
    \node[#1, inner sep=1.5pt] (1) at (0,0.55) {};
\end{pic}
}\hspace{-3pt}}}}
\newcommand{\tinyunit}[1][dot]{
\smash{\raisebox{1pt}{\hspace{-3pt}\ensuremath{\begin{pic}[scale=0.4,yscale=-1]
    \node (0) at (0,0) {};
    \node[#1, inner sep=1.5pt] (1) at (0,0.55) {};
    \draw (0.center) to (1.north);
\end{pic}
}\hspace{-1pt}}}}
\newcommand{\tinycomult}[1][dot]{
\smash{\raisebox{-2pt}{\hspace{-5pt}\ensuremath{\begin{pic}[scale=0.4,yscale=1]
    \node (0) at (0,0) {};
    \node[#1, inner sep=1.5pt] (1) at (0,0.55) {};
    \node (2) at (-0.5,1) {};
    \node (3) at (0.5,1) {};
    \draw (0.center) to (1.center);
    \draw (1.center) to [out=left, in=down, out looseness=1.5] (2.center);
    \draw (1.center) to [out=right, in=down, out looseness=1.5] (3.center);
    \node[#1, inner sep=1.5pt] (1) at (0,0.55) {};
\end{pic}
}\hspace{-3pt}}}}
\newcommand{\tinycounit}[1][dot]{
\smash{\raisebox{-3pt}{\hspace{-3pt}\ensuremath{\begin{pic}[scale=0.4,yscale=1]
    \node (0) at (0,0) {};
    \node[#1, inner sep=1.5pt] (1) at (0,0.55) {};
    \draw (0.center) to (1.south);
\end{pic}
}\hspace{-1pt}}}}
\newcommand\sxto[1]{\mathbin{\smash{
\begin{tikzpicture}[baseline={([yshift=-1pt]
current bounding box.south)}]
    \node (A) at (0,0) [inner xsep=0pt, inner ysep=1pt, minimum width=0.15cm] {\ensuremath{\scriptstyle #1}};
    \draw [->, line width=0.4pt, line cap=round]
        ([xshift=-2.5pt] A.south west)
        to ([xshift=3pt] A.south east);
\end{tikzpicture}}}}
\newenvironment{pic}[1][]
{\begin{aligned}\begin{tikzpicture}[font=\tiny,#1]}
{\end{tikzpicture}\end{aligned}}

\newcommand{\cat}[1]{\ensuremath{\mathbf{#1}}}
\newcommand{\op}{\ensuremath{{}^{\mathrm{op}}}}
\newcommand{\id}[1][]{\ensuremath{\mathrm{id}_{#1}}}
\DeclareMathOperator{\ev}{ev}
\DeclareMathOperator{\st}{st}
\DeclareMathOperator{\dst}{dst}
\DeclareMathOperator{\Tr}{Tr}
\DeclareMathOperator{\Ima}{Im}
\DeclareMathOperator{\FEM}{FEM}
\DeclareMathOperator{\EM}{EM}
\DeclareMathOperator{\Kl}{Kl}
\DeclareMathOperator{\DFMonad}{DFMonad}
\newcommand{\ie}{\textit{i.e.}\xspace}
\newcommand{\eg}{\textit{e.g.}\xspace}

\title{Monads on dagger categories}
\author{Chris Heunen and Martti Karvonen}
\address{School of Informatics, University of Edinburgh\\10 Crichton Street, Edinburgh EH8 9AB, United Kingdom}
\eaddress{\{chris.heunen,martti.karvonen\}@ed.ac.uk}
\copyrightyear{2015}
\begin{document}
\maketitle
\begin{abstract}
  The theory of monads on categories equipped with a dagger (a contravariant identity-on-objects involutive endofunctor) works best when all structure respects the dagger: the monad and adjunctions should preserve the dagger, and the monad and its algebras should satisfy the so-called Frobenius law.
  Then any monad resolves as an adjunction, with extremal solutions given by the categories of Kleisli and Frobenius-Eilenberg-Moore algebras, which again have a dagger.
  We characterize the Frobenius law as a coherence property between dagger and closure, 
  and characterize strong such monads as being induced by Frobenius monoids.
\end{abstract}

\section{Introduction} 

Duality is a powerful categorical notion, and the relationship between (structures in) a category and its dual is fruitful to study.
Especially in self-dual categories, properties may coincide with their dual properties. This article focuses on a specific kind of self-dual category, that has applications in quantum theory~\cite{heunenvicary:cqm} and reversible computing~\cite{heunenkarvonen:reversible,bowmanetal:traced,axelsenkaarsgaard:reversiblerecursion}, amongst others, namely \emph{dagger categories}: categories equipped with a contravariant identity-on-objects involutive endofunctor, called the dagger~\cite{selinger:cpm}. Such categories can behave quite differently than ordinary categories, for example in their limit behaviour~\cite{vicary:daggerlimits}, subobjects~\cite{heunenjacobs:daggerkernels}, additive properties~\cite{heunen:embedding}, or homotopy-theoretical foundations~\cite[9.7]{hott}. As a first step towards a general theory of dagger categories, we study monads on such categories.

A monad on a dagger category is automatically also a comonad. We contend that the theory works best when the monad and comonad satisfy the following \emph{Frobenius law}, depicted in a graphical calculus that will be reviewed in Section~\ref{sec:pictures}. 
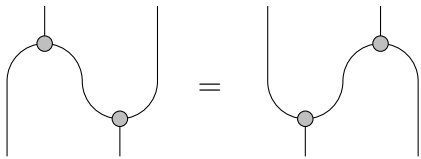
\[
  \begin{pic}
   \draw (0,0) to (0,1) to[out=90,in=180] (.5,1.5) to (.5,2);
   \draw (.5,1.5) to[out=0,in=90] (1,1) to[out=-90,in=180] (1.5,.5) to (1.5,0);
   \draw (1.5,.5) to[out=0,in=-90] (2,1) to (2,2);
   \node[dot] at (.5,1.5) {};
   \node[dot] at (1.5,.5) {};
  \end{pic}
  \quad = \quad
  \begin{pic}[xscale=-1]
   \draw (0,0) to (0,1) to[out=90,in=180] (.5,1.5) to (.5,2);
   \draw (.5,1.5) to[out=0,in=90] (1,1) to[out=-90,in=180] (1.5,.5) to (1.5,0);
   \draw (1.5,.5) to[out=0,in=-90] (2,1) to (2,2);
   \node[dot] at (.5,1.5) {};
   \node[dot] at (1.5,.5) {};
  \end{pic}
\]
This law has the following satisfactory consequences.
\begin{itemize}
  \item Any pair of adjoint functors that also preserve daggers induces a monad satisfying the law; see Section~\ref{sec:adjunctions}.
  \item The Kleisli category of a monad that preserves daggers and satisfies the law inherits a dagger; see Section~\ref{sec:fem}.
  \item For such a monad, the category of those Eilenberg-Moore algebras that satisfy the law inherits a dagger; see Section~\ref{sec:fem}.
  \item In fact, this Kleisli category and Frobenius-Eilenberg-Moore category are the initial and final resolutions of such a monad as adjunctions preserving daggers in the dagger 2-category of dagger categories; see Section~\ref{sec:formal}.
  \item Any monoid in a monoidal dagger category satisfying the law induces a monad satisfying the law; see Section~\ref{sec:monads}.
  \item Moreover, the adjunction between monoids and strong monads becomes an equivalence in the dagger setting; see Section~\ref{sec:strength}.
\end{itemize}
Additionally, Section~\ref{sec:coherence} characterizes the Frobenius law as a natural coherence property between the dagger and closure of a monoidal category. 

Because of these benefits, it is tempting to simply call such monads `dagger monads'. However, many of these results also work without daggers, see~\cite{street:ambidextrous,lauda:ambidextrous,bulacutorrecillas:extensions}. This paper is related to those works, but not a straightforward extension. 
Daggers and monads have come together before in coalgebra~\cite{jacobs:involutive,jacobs:walks}, quantum programming languages~\cite{greenetal:quipper,selingervaliron:lambda}, and matrix algebra~\cite{devosdebaerdemacker:matrix}. The current work differs by taking the dagger into account as a fundamental principle from the beginning.
Finally, Section~\ref{sec:strength} is a noncommutative generalization of~\cite[Theorem~4.5]{pavlovic:abstraction}. It also generalizes the classic Eilenberg-Watts theorem~\cite{watts:eilenberg}, that characterizes certain endofunctors on abelian categories as being of the form $- \otimes B$ for a monoid $B$, to monoidal dagger categories; note that there are monoidal dagger categories that are not abelian~\cite[Appendix~A]{heunen:embedding}.
This article extends an earlier conference proceedings version~\cite{heunenkarvonen:reversible}.
We thank Tom Leinster, who inspired Examples~\ref{ex:dagbiprods} and~\ref{ex:imdagadj}.

\section{Dagger categories}\label{sec:dagger}

We start by introducing dagger categories. They can behave quite differently from ordinary (non-dagger) ones, see \eg~\cite[9.7]{hott}, and are especially useful as semantics for quantum computing~\cite{heunenvicary:cqm}. 

\begin{definition}
  A \emph{dagger} is a functor $\dag\colon\cat{D}\op\to\cat{D}$ satisfying $A^\dag=A$ on objects and $f^{\dag\dag}=f$ on morphisms. A \emph{dagger category} is a category equipped with a dagger. 
\end{definition}

By slight abuse of terminology, we will also call $f^\dag$ the dagger of a morphism $f$.

\begin{example}\label{ex:daggers}
  Dagger categories are plentiful:
  \begin{itemize}
	\item Any groupoid is a dagger category with $f^\dag = f^{-1}$.
    \item Any monoid $M$ equipped with an involutive homomorphism $f\colon M\op\to M$ may be regarded as a one-object dagger category with $x^\dag=f(x)$.
	\item The category \cat{Hilb} of (complex) Hilbert spaces and bounded linear maps is a dagger category, taking the dagger of $f \colon A \to B$ to be its adjoint, i.\ e.\ the unique morphism satisfying $\langle f(a), b \rangle = \langle a, f^\dag(b) \rangle$ for all $a \in A$ and $b \in B$. Write \cat{FHilb} for the full subcategory of finite-dimensional Hilbert spaces; this is a dagger category too.
    \item Write $\cat{Rel}$ for the category with sets as objects and relations $R \subseteq A \times B$ as morphisms $A \to B$; composition is $S \circ R = \{ (a,c) \mid \exists b \in B \colon (a,b) \in R, (b,c) \in S \}$. This becomes a dagger category with $R^\dag=\{(b,a)\mid (a,b)\in R\}$. 
    Equivalently, we may specify $R \subseteq A \times B$ as a function from $A$ to the powerset of $B$.
  \end{itemize}
\end{example}

The guiding principle when working with dagger categories, also known as `the way of the dagger', is that all structure in sight should cooperate with the dagger. For example, more important than isomorphisms are \emph{unitaries}: isomorphisms $f$ whose dagger $f^\dag$ equals its inverse $f^{-1}$. The terminology derives from the category of Hilbert spaces. Similarly, an endomorphism $f$ that equals its own dagger $f^\dag$ is called \emph{self-adjoint}. The following definition provides another example of this motto. 

\begin{definition}
  A \emph{dagger functor} is a functor $F\colon\cat{C}\to \cat{D}$ between dagger categories satisfying $F(f^\dag)=F(f)^\dag$. Denote the category of small dagger categories and dagger functors by \cat{DagCat}. 
\end{definition}

\begin{example} 
  Dagger functors embody various concrete transformations:
  \begin{itemize}
	\item Any functor between groupoids is a dagger functor. 
	\item A functor from a group(oid) to \cat{(F)Hilb} is a dagger functor precisely when it is a \emph{unitary representation}. 
	\item The inclusion $\cat{FHilb}\hookrightarrow \cat{Hilb}$ is a dagger functor. 
  \end{itemize}
\end{example}

There is no need to go further and define `dagger natural transformations': if $\sigma\colon F\to G$ is a natural transformation between dagger functors, then taking daggers componentwise defines a natural transformation $\sigma^\dag\colon G\to F$. Thus the category $[\cat{C},\cat{D}]$ of dagger functors $\cat{C}\to \cat{D}$ and natural transformations is itself a dagger category. 
This implies that \cat{DagCat} is a dagger 2-category, as in the following definition. Notice that products of (ordinary) categories actually provide the category $\cat{DagCat}$ with products, so that the following definition makes sense.

\begin{definition}\label{def:dagtwocat}
  A \emph{dagger 2-category} is a category enriched in \cat{DagCat}, and a \emph{dagger 2-functor} is a $\cat{DagCat}$-enriched functor.
\end{definition}

Another source of examples is given by free and cofree dagger categories~\cite[3.1.17,3.1.19]{heunen:thesis}. Write $\cat{Cat}$ for the category of small categories and functors.

\begin{proposition}\label{prop:cofree}
  The forgetful functor $\cat{DagCat} \to \cat{Cat}$ has a right adjoint $(-)_\leftrightarrows$,
  which sends a category $\cat{C}$ to the full subcategory of $\cat{C}\op \times \cat{C}$ of objects $(A,A)$, 
  and sends a functor $F$ to the restriction of $F\op \times F$.
\end{proposition}
\begin{proof}
  The dagger on $\cat{C}_\leftrightarrows$ is given by $(f,g)^\dag=(g,f)$; this also makes $F_\leftrightarrows$ into a dagger functor. Dagger functors $F \colon \cat{D} \to \cat{C}_\leftrightarrows$ correspond naturally to functors $G \colon \cat{D} \to \cat{C}$ via $Ff=(Gf^\dag,Gf)$, and $Gf=Fh$ when $Ff=(g,h)$.
\end{proof}

We end this section with a useful folklore result.

\begin{lemma}\label{lem:daglem}
  If $F\colon \cat{C}\to \cat{D}$ is a full and faithful functor, then any dagger on $\cat{D}$ induces a unique dagger on $\cat{C}$ such that $F$ is a dagger functor.
\end{lemma}
\begin{proof}
  If $f$ is a morphism in $\cat C$, fullness gives a morphism $f^\dag$ satisfying $F(f)^\dag = F(f^\dag)$, which is unique by faithfulness. This uniqueness also gives $f^{\dag\dag} = f$.
\end{proof}

\section{Graphical calculus}\label{sec:pictures} 

Many proofs in the rest of this paper are most easily presented in graphical form. This section briefly overviews the graphical calculus that governs monoidal (dagger) categories, such as the category $[\cat C, \cat C]$ where our monads will live. For more information, see~\cite{selinger:graphicallanguages}.

\begin{definition}
  A \emph{(symmetric) monoidal dagger category} is a dagger category that is also a (symmetric) monoidal category, satisfying $(f\otimes g)^\dag=f^\dag\otimes g^\dag$ for all morphisms $f$ and $g$, whose coherence maps $\lambda \colon I \otimes A \to A$, $\rho \colon A \otimes I \to A$, and $\alpha \colon (A \otimes B) \otimes C \to A \otimes (B \otimes C)$ (and $\sigma \colon A \otimes B \to B \otimes A$) are unitary. 
\end{definition}

\begin{example}
  Many monoidal structures on dagger categories make them monoidal dagger categories:
  \begin{itemize}
	\item The dagger category $\cat{Rel}$ is a monoidal dagger category under cartesian product.
	\item The dagger category $\cat{(F)Hilb}$ is a monoidal dagger category under tensor product.
	\item For any dagger category \cat{C}, the dagger category $[\cat C, \cat C]$ of dagger functors $\cat C \to \cat C$ is a monoidal dagger category under composition of functors.
	\item Any monoidal groupoid is a monoidal dagger category under $f^\dag = f^{-1}$.
  \end{itemize}
\end{example}

There is a sound and complete graphical calculus for such categories, that represents a morphism $f \colon A \to B$ as 
$\setlength\morphismheight{3mm}\begin{pic}
  \node[morphism,font=\tiny] (f) at (0,0) {$f$};
  \draw (f.south) to +(0,-.1);
  \draw (f.north) to +(0,.1);
\end{pic}$,
and composition, tensor product, and dagger as follows.
\[
  \begin{pic}
    \node[morphism] (f) {$g \circ f$};
    \draw (f.south) to +(0,-.65) node[right] {$A$};
    \draw (f.north) to +(0,.65) node[right] {$C$};
  \end{pic}
  = 
  \begin{pic}
    \node[morphism] (g) at (0,.75) {$g\vphantom{f}$};
    \node[morphism] (f) at (0,0) {$f$};
    \draw (f.south) to +(0,-.3) node[right] {$A$};
    \draw (g.south) to node[right] {$B$} (f.north);
    \draw (g.north) to +(0,.3) node[right] {$C$};
  \end{pic}
  \qquad\qquad
  \begin{pic}
    \node[morphism] (f) {$f \otimes g$};
    \draw (f.south) to +(0,-.65) node[right] {$A \otimes C$};
    \draw (f.north) to +(0,.65) node[right] {$B \otimes D$};
  \end{pic}
  = 
  \begin{pic}
    \node[morphism] (f) at (-.4,0) {$f$};
    \node[morphism] (g) at (.4,0) {$g\vphantom{f}$};
    \draw (f.south) to +(0,-.65) node[right] {$A$};
    \draw (f.north) to +(0,.65) node[right] {$B$};
    \draw (g.south) to +(0,-.65) node[right] {$C$};
    \draw (g.north) to +(0,.65) node[right] {$D$};
  \end{pic}
  \qquad\qquad
  \begin{pic}
    \node[morphism] (f) {$f^\dag$};
    \draw (f.south) to +(0,-.65) node[right] {$B$};
    \draw (f.north) to +(0,.65) node[right] {$A$};
  \end{pic}
  =
  \begin{pic}
    \node[morphism,hflip] (f) {$f$};
    \draw (f.south) to +(0,-.65) node[right] {$B$};
    \draw (f.north) to +(0,.65) node[right] {$A$};
  \end{pic}
\]
Distinguished morphisms are often depicted with special diagrams instead of generic boxes. For example, 
the identity $A \to A$ and the swap map of symmetric monoidal dagger categories are drawn as:
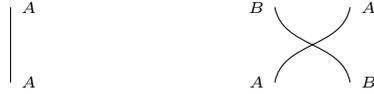
\[\begin{pic}
    \draw (0,0) node[right] {$A$} to (0,1) node[right] {$A$};
  \end{pic}
  \qquad\qquad\qquad
  \begin{pic}
    \draw (0,0) node[left] {$A$} to[out=80,in=-100] (1,1) node[right] {$A$};
    \draw (1,0) node[right] {$B$} to[out=100,in=-80] (0,1) node[left] {$B$};
  \end{pic}
\] 
whereas the (identity on) the monoidal unit object $I$ is drawn as the empty picture: 
\[ 
  \;
\]
  The following definition gives another example: the unit and multiplication of a monoid get a special diagram.

\begin{definition}
  A \emph{monoid} in a monoidal category is an object $A$ with morphisms $\tinymult \colon A \otimes A \to A$ and $\tinyunit \colon I \to A$, satisfying the following equations.
  
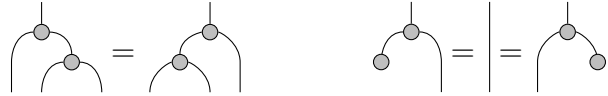
\[
    \begin{pic}[scale=.4]
      \node[dot] (t) at (0,1) {};
      \node[dot] (b) at (1,0) {};
      \draw (t) to +(0,1);
      \draw (t) to[out=0,in=90] (b);
      \draw (t) to[out=180,in=90] (-1,0) to (-1,-1);
      \draw (b) to[out=180,in=90] (0,-1);
      \draw (b) to[out=0,in=90] (2,-1);
    \end{pic}
    =
    \begin{pic}[yscale=.4,xscale=-.4]
      \node[dot] (t) at (0,1) {};
      \node[dot] (b) at (1,0) {};
      \draw (t) to +(0,1);
      \draw (t) to[out=0,in=90] (b);
      \draw (t) to[out=180,in=90] (-1,0) to (-1,-1);
      \draw (b) to[out=180,in=90] (0,-1);
      \draw (b) to[out=0,in=90] (2,-1);
    \end{pic}
  \qquad\qquad
  \begin{pic}[scale=.4]
    \node[dot] (d) {};
    \draw (d) to +(0,1);
    \draw (d) to[out=0,in=90] +(1,-1) to +(0,-1);
    \draw (d) to[out=180,in=90] +(-1,-1) node[dot] {};
  \end{pic}
  =
  \begin{pic}[scale=.4]
    \draw (0,0) to (0,3);
  \end{pic}
  =
  \begin{pic}[yscale=.4,xscale=-.4]
    \node[dot] (d) {};
    \draw (d) to +(0,1);
    \draw (d) to[out=0,in=90] +(1,-1) to +(0,-1);
    \draw (d) to[out=180,in=90] +(-1,-1) node[dot] {};
  \end{pic}
  \]
  A \emph{comonoid} in a monoidal category is a monoid in the opposite category; an object $A$ with morphisms $\tinycomult \colon A \to A \otimes A$ and $\tinycounit \colon A \to I$ satisfying the duals of the above equations.
  A monoid in a symmetric monoidal category is \emph{commutative} if it satisfies:
  \[\begin{pic}[scale=.4]
      \node[dot] (d) at (0,0) {};
      \draw (d.north) to +(0,1);
      \draw (d.west) to[out=180,in=90] +(-.75,-1) to +(0,-1);
      \draw (d.east) to[out=0,in=90] +(.75,-1) to +(0,-1);
    \end{pic}
    \quad = \quad
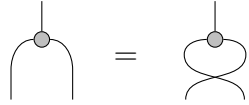
    \begin{pic}[scale=.4]
      \node[dot] (d) at (0,0) {};
      \draw (d.north) to +(0,1);
      \draw (d.west) to[out=180,in=90] +(-.75,-.5) to[out=-90,in=90] +(2,-1.5);
      \draw (d.east) to[out=0,in=90] +(.75,-.5) to[out=-90,in=90] +(-2,-1.5);
    \end{pic}
  \]
  A monoid in a monoidal dagger category is a \emph{dagger Frobenius monoid} if it satisfies the following \emph{Frobenius law}.
  
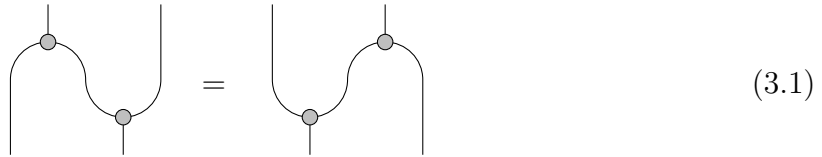
\begin{equation}\label{eq:frobeniuslaw}
  \begin{pic}
   \draw (0,0) to (0,1) to[out=90,in=180] (.5,1.5) to (.5,2);
   \draw (.5,1.5) to[out=0,in=90] (1,1) to[out=-90,in=180] (1.5,.5) to (1.5,0);
   \draw (1.5,.5) to[out=0,in=-90] (2,1) to (2,2);
   \node[dot] at (.5,1.5) {};
   \node[dot] at (1.5,.5) {};
  \end{pic}
  \quad = \quad
  \begin{pic}[xscale=-1]
   \draw (0,0) to (0,1) to[out=90,in=180] (.5,1.5) to (.5,2);
   \draw (.5,1.5) to[out=0,in=90] (1,1) to[out=-90,in=180] (1.5,.5) to (1.5,0);
   \draw (1.5,.5) to[out=0,in=-90] (2,1) to (2,2);
   \node[dot] at (.5,1.5) {};
   \node[dot] at (1.5,.5) {};
  \end{pic}
 \end{equation}
\end{definition}

The Frobenius law might look mysterious, but will turn out to be precisely the right property to make monads respect daggers. 
Section~\ref{sec:coherence} below will formally justify it as a coherence property between closure and the dagger. 
For now we illustrate that the Frobenius law corresponds to natural mathematical structures in example categories.

\begin{example}\label{ex:groupoidFrob} 
  See~\cite{heunencontrerascattaneo:groupoids,heunentull:groupoids,vicary:quantumalgebras} for more information on the following examples.
  \begin{itemize} 
    \item Let $\cat{G}$ be a small groupoid, and $G$ its set of objects. The assignments 
    \begin{align*}
      \{*\}&\mapsto \{\id[A]\mid A\in G\} 
      &
      (f,g)&\mapsto 
      \begin{cases} 
        \{f\circ g\} &\text{ if }f\circ g\text{ is defined} \\ 
        \emptyset &\text{ otherwise} 
      \end{cases} 
      \\
    \intertext{define a dagger Frobenius monoid in $\cat{Rel}$ on the set of morphisms of $\cat{G}$.
    Conversely, any dagger Frobenius monoid in $\cat{Rel}$ is of this form.
    \item Let $\cat{G}$ be a finite groupoid, and $G$ its set of objects. The assignments}
      1 &\mapsto \sum_{A\in G}\id[A] 
      &
      f\otimes g & \mapsto 
      \begin{cases} 
        f\circ g &\text{ if }f\circ g\text{ is defined} \\ 
        0 &\text{ otherwise} 
      \end{cases} 
    \end{align*}
    define a dagger Frobenius monoid in \cat{(F)Hilb} on the Hilbert space of which the morphisms of $\cat{G}$ form an orthonormal basis.
    Conversely, any dagger Frobenius monoid in $\cat{(F)Hilb}$ is of this form.
  \end{itemize}
\end{example}

The following lemma exemplifies graphical reasoning. Recall that a (co)monoid homomorphism is a morphism $f$ between (co)monoids satisfying $f \circ \tinyunit = \tinyunit$ and $f \circ \tinymult = \tinymult \circ (f \otimes f)$ ($\tinycounit \circ f = \tinycounit$ and $\tinycomult \circ f = (f \otimes f) \circ \tinycomult$).

\begin{lemma}\label{lem:strictmorphismsareiso}
  A monoid homomorphism between dagger Frobenius monoids in a monoidal dagger category, that is also a comonoid homomorphism, is an isomorphism.
\end{lemma}
\begin{proof}
  Construct an inverse to $A \sxto{f} B$ as follows:
  \[\begin{pic}[xscale=.75,yscale=.5]
      \node (f) [morphism] at (0,0) {$f$};
      \draw (f.north)
      to [out=up, in=right] +(-0.5,0.5) node (d) [dot] {}
      to [out=left, in=up] +(-0.5,-0.5)
      to [out=down, in=up] +(0,-1.8) node [below] {$B$};
      \draw (d.north) to +(0,0.4) node [dot] {};
      \draw (f.south)
      to [out=down, in=left] +(0.5,-0.5) node (d) [dot] {}
      to [out=right, in=down] +(0.5,0.5)
      to [out=up, in=down] +(0,1.8) node [above] {$A$};
      \draw (d.south) to +(0,-0.4) node [dot] {};
    \end{pic}
  \]
  The composite with $f$ gives the identity in one direction:
  
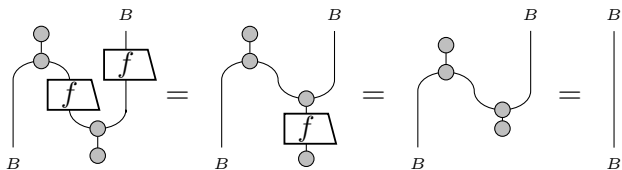
\[
    \begin{pic}[xscale=.75,yscale=.5]
      \node (f) [morphism] at (0,0) {$f$};
      \draw (f.north)
      to [out=up, in=right] +(-0.5,0.5) node (d) [dot] {}
      to [out=left, in=up] +(-0.5,-0.5)
      to [out=down, in=up] +(0,-1.8) node [below] {$B$};
      \draw (d.north) to +(0,0.5) node [dot] {};
      \draw (f.south)
      to [out=down, in=left] +(0.5,-0.5) node (d) [dot] {}
      to [out=right, in=down] +(0.5,0.5)
      to [out=down, in=down] +(0,0.8)
      node (f2) [morphism, anchor=south, width=0.2cm]
      {$f$};
      \draw (d.south) to +(0,-0.5) node [dot] {};
      \draw (f2.north) to +(0,0.5) node [above] {$B$};
    \end{pic}
    =
    \begin{pic}[xscale=.75,yscale=.5]
      \draw (0,0)
      to [out=up, in=right] +(-0.5,0.5) node (d) [dot] {}
      to [out=left, in=up] +(-0.5,-0.5)
      to [out=down, in=up] +(0,-1.8) node [below] {$B$};
      \draw (d.north) to +(0,0.5) node [dot] {};
      \node (f) [morphism] at (0.5,-1.3) {$f$};
      \draw (0,0)
      to [out=down, in=left] +(0.5,-0.5) node (e) [dot] {}
      to [out=right, in=down] +(0.5,0.5)
      to [out=up, in=down] +(0,1.3) node [above] {$B$};
      \draw (f.north) to (e.south);
      \draw (f.south) to +(0,-0.4) node [dot] {};
    \end{pic}
    =
    \begin{pic}[xscale=.75,yscale=.5]
      \draw (0,0)
      to [out=up, in=right] +(-0.5,0.5) node (d) [dot] {}
      to [out=left, in=up] +(-0.5,-0.5)
      to [out=down, in=up] +(0,-1.5) node [below] {$B$};
      \draw (d.north) to +(0,0.5) node [dot] {};
      \draw (0.5,-0.5) to +(0,-0.5) node [dot] {};
      \draw (0,0)
      to [out=down, in=left] +(0.5,-0.5) node (d) [dot] {}
      to [out=right, in=down] +(0.5,0.5)
      to [out=up, in=down] +(0,1.6) node [above] {$B$};
    \end{pic}
    =
    \begin{pic}[yscale=.5]
      \draw (0,0) node [below] {$B$} to +(0,3.1) node [above] {$B$};
    \end{pic}
  \]
  The third equality uses the Frobenius law~\eqref{eq:frobeniuslaw} and the unit law.
  The other composite is the identity by a similar argument.
\end{proof}

\section{Dagger adjunctions}\label{sec:adjunctions}

This section considers adjunctions that respect daggers.

\begin{definition} 
  A \emph{dagger adjunction} is an adjunction between dagger categories where both functors are dagger functors.
\end{definition}

Note that the previous definition did not need to specify left and right adjoints, because the dagger makes the adjunction go both ways. If $F\colon \cat C \to \cat D$ and $G \colon \cat D \to \cat C$ are dagger adjoints, say $F \dashv G$ with natural bijection $\theta \colon \cat{D}(FA,B) \to \cat{C}(A,GB)$, then $f \mapsto \theta(f^\dag)^\dag$ is a natural bijection $\cat{D}(A,FB) \to \cat{C}(GA,B)$, 
whence also $G \dashv F$.

For example, a dagger category $\cat{C}$ has a zero object if and only if the unique dagger functor $\cat{C}\to\cat{1}$ has a dagger adjoint. Here, a \emph{zero object} is one that is both initial and terminal, and hence induces zero maps between any two objects. This is the nullary version of the following example: a product $\smash{A \stackrel{p_A}{\longleftarrow} A \times B \stackrel{p_B}{\longrightarrow} B}$ is a \emph{dagger biproduct} when $p_A \circ p_A^\dag = \id$, $p_A \circ p_B^\dag = 0$, $p_B \circ p_A^\dag = 0$, $p_B \circ p_B^\dag = \id$.

\begin{example}\label{ex:dagbiprods}
  A dagger category $\cat{C}$ with zero object has binary dagger biproducts if and only if the diagonal functor $\cat{C} \to \cat{C} \times \cat{C}$ has a dagger adjoint $(-)\oplus (-)$ with dagger epic counit $(p_A,p_B)\colon (A\oplus B,A\oplus B)\to (A,B)$, \textit{i.e.}\ $(p_A \circ p_A^\dag,p_B\circ p_B^\dag)=(\id[A],\id[B])$ for $A,B \in \cat{C}$.
\end{example}
\begin{proof}
  The implication from left to right is routine. For the other direction, a right adjoint to the diagonal is well-known to fix binary products~\cite[V.5]{maclane:categories}. 
  If it additionally preserves daggers the product is also a coproduct, so it remains to check that the required equations governing $p_A$ and $p_B$ are satisfied. By naturality, the diagram
  \[\begin{tikzpicture}[xscale=3.5,yscale=1.5]
    \node (tl) at (0,1) {$A$};
    \node (t) at (1,1) {$A \oplus B$};
    \node (tr) at (2,1) {$B$};
    \node (bl) at (0,0) {$A$};
    \node (b) at (1,0) {$A \oplus B$};
    \node (br) at (2,0) {$B$};
    \draw[->] (tl) to node[above] {$p_A^\dag$} (t);
    \draw[->] (t) to node[above] {$p_B$} (tr);
    \draw[->] (bl) to node[below] {$p_A^\dag$} (b);
    \draw[->] (b) to node[below] {$p_B$} (br);
    \draw[->] (tl) to node[left] {$\id$} (bl);
    \draw[->] (tr) to node[right] {$0$} (br);
    \draw[->] (t) to node[right] {$\id\oplus 0$} (b);
  \end{tikzpicture}\]
  commutes, so that $p_B \circ p_A^\dag = 0$. By symmetry $p_A \circ p_B^\dag = 0$, and the remaining equations hold by assumption.
\end{proof}

Here is a more involved example of a dagger adjunction.

\begin{example}\label{ex:imdagadj}
  The monoids $(\mathbb{N},+)$ and $(\mathbb{Z},+)$ become one-object dagger categories under the trivial dagger $k\mapsto k$. The inclusion $\mathbb{N}\hookrightarrow \mathbb{Z}$ is a dagger functor. It induces a dagger functor $F\colon [\mathbb{Z},\cat{FHilb}]\to [\mathbb{N},\cat{FHilb}]$, which has a dagger adjoint $G$. 
\end{example}
\begin{proof}
  An object of $[\mathbb{Z},\cat{FHilb}]$ is a self-adjoint isomorphism $T\colon A\to A$ on a finite-dimensional Hilbert space $A$, whereas an object of $[\mathbb{N},\cat{FHilb}]$ is a just a self-adjoint morphism $T \colon A \to A$ in $\cat{FHilb}$.
  To define $G$ on objects, notice that a self-adjoint morphism $T \colon A \to A$ restricts to a self-adjoint surjection from $\ker(T)^\perp = \overline{\Ima T}$ to itself, and by finite-dimensionality of $A$ hence to a self-adjoint isomorphism $G(T)$ on $\Ima T$.

  On a morphism $f \colon T \to S$ in $[\mathbb{N},\cat{FHilb}]$, define $Gf$ to be the restriction of $f$ to $\Ima T$. To see this is well-defined, \ie the right diagram below commutes if the left one does,
  \[
    \begin{aligned}\begin{tikzpicture}
     \matrix (m) [matrix of math nodes,row sep=2em,column sep=4em,minimum width=2em]
     {A & B \\
      A & B \\};
     \path[->]
     (m-1-1) edge node [left] {$T$} (m-2-1)
             edge node [above] {$f$} (m-1-2)
     (m-2-1) edge node [below] {$f$} (m-2-2)
     (m-1-2) edge node [right] {$S$} (m-2-2);
    \end{tikzpicture}\end{aligned}
    \qquad \implies \qquad
    \begin{aligned}\begin{tikzpicture}
     \matrix (m) [matrix of math nodes,row sep=2em,column sep=4em,minimum width=2em]
     {\Ima  T & \Ima S \\
      \Ima T & \Ima S \\};
     \path[->]
     (m-1-1) edge node [left] {$T$} (m-2-1)
             edge node [above] {$G(f)$} (m-1-2)
     (m-2-1) edge node [below] {$G(f)$} (m-2-2)
     (m-1-2) edge node [right] {$S$} (m-2-2);
    \end{tikzpicture}\end{aligned}
  \] 
  observe that if $b\in\Ima T$ then $b=T(a)$ for some $b\in H$, so that $f(b)=fT(a)=Sf(a)$, and hence $f(b)\in \Ima S$. 
  This definition of $G$ is easily seen to be dagger functorial.

  To prove that $F$ and $G$ are dagger adjoint, it suffices to define a natural transformation $\eta \colon \id \to F \circ G$, because $G \circ F$ is just the identity. Define $\eta_T$ to be the projection $A \to \Ima T$, which is a well-defined morphism in $[\mathbb{N},\cat{FHilb}]$:
  \[
    \begin{tikzpicture}
     \matrix (m) [matrix of math nodes,row sep=2em,column sep=4em,minimum width=2em]
     {A & \Ima T \\
      A & \Ima T\rlap{.} \\};
     \path[->]
     (m-1-1) edge node [left] {$T$} (m-2-1)
             edge node [above] {$\eta_T$} (m-1-2)
     (m-2-1) edge node [below] {$\eta_T$} (m-2-2)
     (m-1-2) edge node [right] {$T$} (m-2-2);
     \end{tikzpicture}
  \] 
  Naturality of $\eta$ boils down to commutativity of 
  \[
  \begin{tikzpicture}
     \matrix (m) [matrix of math nodes,row sep=2em,column sep=4em,minimum width=2em]
     {
      A & \Ima T & \\
      B & \Ima S \\};
     \path[->]
     (m-1-1) edge node [left] {$f$} (m-2-1)
             edge node [above] {$\eta_T$} (m-1-2)
     (m-2-1) edge node [below] {$\eta_S$} (m-2-2)
     (m-1-2) edge node [right] {$f$} (m-2-2);
   \end{tikzpicture}
   \] 
   which is easy to verify.
\end{proof}

There are variations on the previous example. For example, $(n,m)\mapsto (m,n)$ induces daggers on $\mathbb{N}\times\mathbb{N}$ and $\mathbb{Z}\times\mathbb{Z}$. A dagger functor $\mathbb{N}\times\mathbb{N}\to \cat{FHilb}$ corresponds to a choice of a normal map, which again restricts to a normal isomorphism on its image. This defines a dagger adjoint to the inclusion $[\mathbb{N}\times\mathbb{N},\cat{FHilb}]\to [\mathbb{Z}\times\mathbb{Z},\cat{FHilb}]$. 

Recall that $F \colon \cat{C} \to \cat{D}$ is a \emph{Frobenius functor} when it has a left adjoint $G$ that is simultaneously right adjoint. This is also called an \emph{ambidextrous adjunction}~\cite{lauda:ambidextrous}.

\begin{proposition} 
  If $F$ is a Frobenius functor with adjoint $G$, then $F_\leftrightarrows$ and $G_\leftrightarrows$ as in Proposition~\ref{prop:cofree} are dagger adjoint.
\end{proposition}
\begin{proof}
  If $F \colon \cat C \to \cat D$ there is a natural bijection
  \begin{align*}
    \cat{C}_\leftrightarrows\big( (A,A) , G_\leftrightarrows(B,B) \big)
    & = \cat{C}(GB,A) \times \cat{C}(A,GB) \\
    & \cong \cat{D}(B,FA) \times \cat{C}(FA,B) \\
    & = \cat{D}_\leftrightarrows\big( F_\leftrightarrows(A,A), (B,B) \big)
  \end{align*}
  because $G \dashv F \dashv G$.
\end{proof}

\section{Frobenius monads}\label{sec:monads}

We now come to our central notion: the \emph{Frobenius law} for monads. It is the dagger version of a similar notion in~\cite{street:ambidextrous}. The monads of~\cite{street:ambidextrous} correspond to ambijuctions, whereas our monads correspond to dagger adjunctions.

\begin{definition}
  A \emph{dagger Frobenius monad} on a dagger category $\cat{C}$ is a dagger Frobenius monoid in $[\cat{C},\cat{C}]$;
  explicitly, a monad $(T,\mu,\eta)$ on $\cat{C}$ with $T(f^\dag)=T(f)^\dag$ and 
  \begin{equation}\label{eq:frobeniusformonads}
    T(\mu_A) \circ \mu^\dag_{T(A)} = \mu_{T(A)} \circ T(\mu_A^\dag).
  \end{equation}
\end{definition}

The following family is our main source of examples of dagger Frobenius monads.
We will see in Example~\ref{ex:measurement} below that it includes quantum measurement.

\begin{example}\label{example:tensor}
   A monoid $(B,\tinymult,\tinyunit)$ in a monoidal dagger category $\cat{C}$ is a dagger Frobenius monoid if and only if the monad $- \otimes B \colon \cat C \to \cat{C}$ is a dagger Frobenius monad.
\end{example}
\begin{proof}
   The monad laws become the monoid laws.
   \[
     \mu_A = \begin{pic}
       \node[dot] (d) at (0,0) {};
       \draw (d.north) to (0,.5) node[right] {$B$};
       \draw (d.east) to[out=0,in=90] +(.3,-.3) to +(0,-.2) node[right] {$B$};
       \draw (d.west) to[out=180,in=90] +(-.3,-.3) to +(0,-.2) node[right] {$B$};
       \draw (-.8,-.5) node[left] {$A$} to (-.8,.5) node[left] {$A$};
     \end{pic} 
     \qquad
     \eta_A = \begin{pic}
       \node[dot] (d) at (0,0) {};
       \draw (d.north) to (0,.5) node[right] {$B$};
       \draw (-.4,-.5) node[left] {$A$} to (-.4,.5) node[left] {$A$};
     \end{pic} 
  \]
  The Frobenius law of the monoid implies the Frobenius law of the monad:
  \[
    T\mu \circ \mu^\dag_T
    =
    \begin{pic}[scale=.75]
      \draw (0,0) node[below] {$B$} to (0,1) to[out=90,in=180] (.5,1.5) to (.5,2) node[above] {$B$};
      \draw (.5,1.5) to[out=0,in=90] (1,1) to[out=-90,in=180] (1.5,.5) to (1.5,0) node[below] {$B$};
      \draw (1.5,.5) to[out=0,in=-90] (2,1) to (2,2) node[above] {$B$};
      \node[dot] at (.5,1.5) {};
      \node[dot] at (1.5,.5) {};
      \draw (-.5,0) node[below] {$A$} to (-.5,2) node[above] {$A$};
    \end{pic}
    =
    \begin{pic}[yscale=.75,xscale=-.75]
      \draw (0,0) node[below] {$B$} to (0,1) to[out=90,in=180] (.5,1.5) to (.5,2) node[above] {$B$};
      \draw (.5,1.5) to[out=0,in=90] (1,1) to[out=-90,in=180] (1.5,.5) to (1.5,0) node[below] {$B$};
      \draw (1.5,.5) to[out=0,in=-90] (2,1) to (2,2) node[above] {$B$};
      \node[dot] at (.5,1.5) {};
      \node[dot] at (1.5,.5) {};
      \draw (2.5,0) node[below] {$A$} to (2.5,2) node[above] {$A$};
    \end{pic}
    =
    \mu_T \circ T\mu^\dag.
  \]
  The converse follows by taking $A=I$.
\end{proof}

For another example: if $T$ is a dagger Frobenius monad on a dagger category $\cat{C}$, and $\cat{D}$ is any other dagger category, then $T \circ -$ is a dagger Frobenius monad on $[\cat D, \cat C]$.

\begin{lemma}\label{lem:extendedfrobeniuslaw}
  If $T$ is a dagger Frobenius monad on a dagger category, $\mu^\dag \circ \mu = \mu T \circ T\mu^\dag$.
\end{lemma}
\begin{proof}
  The following graphical derivation holds for any dagger Frobenius monoid.
  \begin{align*}
    \begin{pic}[yscale=.5625,xscale=.75]
      \node (0a) at (-0.5,-1) {};
      \node (0b) at (0.5,-1) {};
      \node[dot] (1) at (0,1) {};
      \node[dot] (2) at (0,2) {};
      \node (3a) at (-0.5,3) {};
      \node (3b) at (0.5,3) {};
      \draw[out=90,in=180,looseness=.66] (0a) to (1.west);
      \draw[out=90,in=0,looseness=.66] (0b) to (1.east);
      \draw (1.north) to (2.south);
      \draw[out=180,in=270] (2.west) to (3a);
      \draw[out=0,in=270] (2.east) to (3b);
    \end{pic}
    =
    \begin{pic}[yscale=.5625,xscale=.75]
      \node (0d) at (-1,-1) {};
      \node [dot] (0c) at (-1.5,1) {};
      \node [dot] (0a) at (-1,0) {};
      \node (0b) at (0.5,-1) {};
      \node[dot] (1) at (0,1) {};
      \node[dot] (2) at (0,2) {};
      \node (3a) at (-0.5,3) {};
      \node (3b) at (0.5,3) {};
      \draw (0a) to (0d);
      \draw[out=180,in=270] (0a) to (0c);
      \draw[out=0,in=180] (0a) to (1.west);
      \draw[out=90,in=0,looseness=0.66] (0b) to (1.east);
      \draw (1.north) to (2.south);
      \draw[out=180,in=270] (2.west) to (3a);
      \draw[out=0,in=270] (2.east) to (3b);
    \end{pic}
    =
    \begin{pic}[yscale=.5625,xscale=.75]
      \node (0d) at (-2,-1) {};
      \node [dot] (0c) at (-1.5,2) {};
      \node [dot] (0a) at (-.5,0) {};
      \node (0b) at (-.5,-1) {};
      \node[dot] (1) at (-1.5,1) {};
      \node[dot] (2) at (0,2) {};
      \node (3a) at (-0.5,3) {};
      \node (3b) at (0.5,3) {};
      \draw[out=180,in=90,looseness=0.66] (1.west) to (0d);
      \draw (1.north) to (0c.south);
      \draw[out=180,in=0,looseness=1] (0a.west) to (1.east);
      \draw (0b) to (0a.south);
      \draw[out=0,in=270,out looseness=.7] (0a.east) to (2.south);
      \draw[out=180,in=270] (2.west) to (3a);
      \draw[out=0,in=270] (2.east) to (3b);
    \end{pic}
    =
    \begin{pic}[yscale=.5625,xscale=.75]
      \node (0d) at (-2,-1) {};
      \node [dot] (0c) at (-1.5,2.5) {};
      \node [dot] (0a) at (-.75,.75) {};
      \node (0b) at (0,-1) {};
      \node[dot] (1) at (-1.5,1.5) {};
      \node[dot] (2) at (0,0) {};
      \node (3a) at (0,3) {};
      \node (3b) at (0.75,3) {};
      \draw[out=180,in=90,out looseness=.5] (1.west) to (0d);
      \draw (1.north) to (0c.south);
      \draw[out=180,in=0] (0a.west) to (1.east);
      \draw (0b) to (2.south);
      \draw[out=270,in=180] (0a.south) to (2.west);
      \draw[out=0,in=270,out looseness=.66] (2.east) to (3b);
      \draw[out=0,in=270,out looseness=.8] (0a.east) to (3a);
    \end{pic}
    =
    \begin{pic}[yscale=.75,xscale=.75]
      \node [dot] (a) at (-0.5, 0.5) {};
      \node [dot] (b) at (0,0) {};
      \node [dot] (c) at (1,1) {};
      \node [dot] (d) at (2,0) {};
      \node (A) at (0,-1) {};
      \node (B) at (2,-1) {};
      \node (C) at (1,2) {};
      \node (D) at (2.5,2) {};
      \draw (c.north) to (C);
      \draw (B) to (d.south);
      \draw (A) to (b.south);
      \draw[out=270,in=180] (a.south) to (b.west);
      \draw[out=0,in=180] (b.east) to (c.west);
      \draw[out=0,in=180] (c.east) to (d.west);
      \draw[out=0,in=270,out looseness=.66] (d.east) to (D);
    \end{pic}
    =
    \begin{pic}[yscale=.75,xscale=.75]
      \node (0) at (0,0) {};
      \node (0a) at (0,1) {};
      \node[dot] (1) at (0.5,2) {};
      \node[dot] (2) at (1.5,1) {};
      \node (3) at (1.5,0) {};
      \node (4) at (2,3) {};
      \node (4a) at (2,2) {};
      \node (5) at (0.5,3) {};
      \draw (0) to (0a.center);
      \draw[out=90,in=180] (0a.center) to (1.west);
      \draw[out=0,in=180] (1.east) to (2.west);
      \draw[out=0,in=270,looseness=.66] (2.east) to (4a.center);
      \draw (4a.center) to (4);
      \draw (2.south) to (3);
      \draw (1.north) to (5);
    \end{pic}
  \end{align*}
  These equalities use the unit law, the Frobenius law, and associativity.
\end{proof}

The following lemma shows that dagger Frobenius monads have the same relationship to dagger adjunctions as ordinary monads have to ordinary adjunctions.

\begin{lemma}\label{lem:daggeradjunction}
  If $F \dashv G$ is a dagger adjunction, then $G \circ F$ is a dagger Frobenius monad.
\end{lemma} 
\begin{proof}
  It is clear that $T = G \circ F$ is a dagger functor.
  The Frobenius law follows from applying~\cite[Corollary 2.22]{lauda:ambidextrous} to $\cat{DagCat}$.
  We will be able to give a self-contained proof after Theorem~\ref{thm:comparison} below.
\end{proof}

For example, in \cat{Rel} and \cat{Hilb}, the dagger biproduct monad induced by the dagger adjunction of Example~\ref{ex:dagbiprods} is of the form $-\otimes (I\oplus I)$ as in Example~\ref{example:tensor}. 
However, not all dagger Frobenius monads are of this form: the Frobenius monad induced by the dagger adjunction of Example~\ref{ex:imdagadj} in general decreases the dimension of the underlying space, and hence cannot be of the form $-\otimes B$ for a fixed $B$.

\section{Algebras}\label{sec:fem} 

Next we consider algebras for dagger Frobenius monads. We start by showing that Kleisli categories of dagger Frobenius monads inherit a dagger.

\begin{lemma}\label{lem:kleislidagger}
  If $T$ is a dagger Frobenius monad on a dagger category $\cat{C}$, then $\Kl(T)$ carries a dagger that commutes with the canonical functors $\Kl(T) \to \cat{C}$ and $\cat{C}\to\Kl(T)$.
\end{lemma}
\begin{proof} 
  A straightforward calculation establishes that
  \[
    \big(A \sxto{f} T(B)\big)
    \;\mapsto\;
    \big(B \sxto{\eta} T(B) \sxto{\mu^\dag} T^2(B) \sxto{T(f^\dag)} T(A)\big)
  \]
  is a dagger on $\Kl(T)$ commuting with the functors $\cat{C}\to \Kl(T)$ and $\Kl(T)\to \cat{C}$.
\end{proof} 

If we want algebras to form a dagger category, it turns out that the category of all Eilenberg-Moore algebras is too large. The crucial law is the following \emph{Frobenius law} for algebras.

\begin{definition} 
   Let $T$ be a monad on a dagger category $\cat{C}$. A \emph{Frobenius-Eilenberg-Moore algebra}, or \emph{FEM-algebra} for short, is an Eilenberg-Moore algebra $a \colon T(A) \to A$ that makes the following diagram commute.
   \begin{equation}\label{eq:femlaw}\begin{aligned}\begin{tikzpicture}
     \matrix (m) [matrix of math nodes,row sep=2em,column sep=4em,minimum width=2em]
     {
      T(A) & T^2(A) & \\
      T^2(A) & T(A) \\};
     \path[->]
     (m-1-1) edge node [left] {$\mu^\dag$} (m-2-1)
             edge node [above] {$T(a)^\dag$} (m-1-2)
     (m-2-1) edge node [below] {$T(a)$} (m-2-2)
     (m-1-2) edge node [right] {$\mu$} (m-2-2);
   \end{tikzpicture}\end{aligned}\end{equation}   
   Denote the category of FEM-algebras $(A,a)$ and algebra homomorphisms by $\FEM(T)$.
\end{definition}

When the dagger Frobenius monad is of the form $T(A)=A \otimes B$ as in Example~\ref{example:tensor}, the Frobenius law~\eqref{eq:femlaw} for an algebra $a \colon T(A) \to A$ becomes the following equation, which resembles the Frobenius law~\eqref{eq:frobeniuslaw} for monoids and monads.
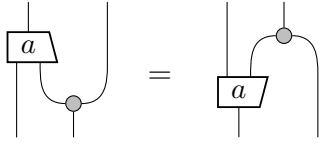
\[
 \begin{pic}[xscale=1.5]
   \node[morphism, minimum width=25mm] (a)  at (0,0.75) {$a$};
   \node[dot] (b) at (.4,0) {};
   \draw ([xshift=-.5mm]a.south west) to +(0,-1);
   \draw (a.north) to +(0,.4);
   \draw (b) to[out=0,in=-90] +(.3,.5) to +(0,.85);
   \draw ([xshift=.5mm]a.south east) to[out=-90,in=180]  (b);
   \draw (b) to +(0,-.45);
 \end{pic}
 \quad = \quad
 \begin{pic}[xscale=1.5]
   \node[morphism,hflip, minimum width=25mm] (a)  at (0,0) {$a$};
   \node[dot] (b) at (.4,0.75) {};
   \draw ([xshift=-.5mm]a.north west) to +(0,1);
   \draw (a.south) to +(0,-.4);
   \draw (b) to[out=0,in=90] +(.3,-.5) to +(0,-.85);
   \draw ([xshift=.5mm]a.north east) to[out=90,in=180]  (b);
   \draw (b) to +(0,.45);
 \end{pic} 
\]

\begin{example}\label{ex:kleislifem}
  Any free algebra $\mu_A \colon T^2(A) \to T(A)$ of a dagger Frobenius monad $T$ on a dagger category $\cat C$ is a FEM-algebra.
  Hence there is an embedding $\Kl(T) \to \FEM(T)$.
\end{example}
\begin{proof}
  The Frobenius law for the free algebra is the Frobenius law of the monad.
\end{proof}

There are many EM-algebras that are not FEM-algebras; a family of examples can be derived from~\cite[Theorem~6.4]{pavlovic:abstraction}. Here is a concrete example.

\begin{example}\label{ex:emnonfem}
  The complex $n$-by-$n$-matrices form a Hilbert space $A$ with inner product $\left\langle a,b\right\rangle=\tfrac{1}{n}\Tr (a^\dag \circ b)$. Matrix multiplication $m\colon A\otimes A\to A$ makes $A$ into a dagger Frobenius monoid in $\cat{FHilb}$, and hence $T=-\otimes A$ into a dagger Frobenius monad on $\cat{FHilb}$. Define a monoid homomorphism $U\colon A\to A$ by conjugation $a\mapsto u^\dag\circ a\circ u$ with a unitary matrix $u \in A$. Then $m\circ (\id[A] \otimes U)$ is an EM-algebra that is a FEM-algebra if and only if $u=u^\dag$.
\end{example}
\begin{proof}
  Because $U^\dag=u \circ - \circ u^\dag$, the Frobenius law~\eqref{eq:femlaw} for $T$ unfolds to the following.
  
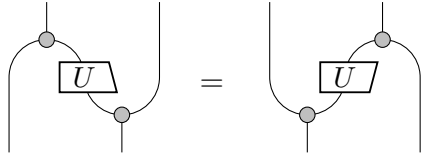
\[
    \begin{pic}
     \draw (0,0) to (0,1) to[out=90,in=180] (.5,1.5) to (.5,2);
     \draw (.5,1.5) to[out=0,in=90] (1,1) to[out=-90,in=180] (1.5,.5) to (1.5,0);
     \draw (1.5,.5) to[out=0,in=-90] (2,1) to (2,2);
     \node[morphism] at (1,1) {$U$};
     \node[dot] at (.5,1.5) {};
     \node[dot] at (1.5,.5) {};
    \end{pic}
    \quad = \quad
    \begin{pic}[xscale=-1]
     \draw (0,0) to (0,1) to[out=90,in=180] (.5,1.5) to (.5,2);
     \draw (.5,1.5) to[out=0,in=90] (1,1) to[out=-90,in=180] (1.5,.5) to (1.5,0);
     \draw (1.5,.5) to[out=0,in=-90] (2,1) to (2,2);
     \node[morphism,hflip] at (1,1) {$U$};
     \node[dot] at (.5,1.5) {};
     \node[dot] at (1.5,.5) {};
    \end{pic}
  \]
  This comes down to $U=(U^*)^\dag$, that is, $u=u^\dag$. 
\end{proof}

Before we calculate $\FEM(-\otimes B)$ for a dagger Frobenius monoid $B$ induced by an arbitrary groupoid, we work out an important special case.

\begin{example}\label{ex:measurement}
  Let $B$ be a dagger Frobenius monoid in $\cat{FHilb}$ induced by a finite discrete groupoid $\cat{G}$ as in Example~\ref{ex:groupoidFrob}. A FEM-algebra structure on a Hilbert space $A$ for $- \otimes B$ consists of \emph{quantum measurements} on $A$: orthogonal projections on $A$ that sum to the identity.
\end{example} 
For more information about quantum measurements, see~\cite[Section~3.2]{heinosaariziman:quantum}.
\begin{proof}
  A FEM-algebra structure on $A$ consists of a map $a\colon A\otimes B\to A$ subject to the FEM-laws. Since $B$ has a basis indexed by objects of \cat{G}, it suffices to understand the maps $P_G\colon A\to A$ defined by $v\mapsto a(v\otimes \id[G])$. The associative law implies that each $P_G$ satisfies $P_G\circ P_G=P_G$, and from the Frobenius law we get that each $P_G$ is self-adjoint, so that each $P_G$ is an orthonormal projection. The unit law says that $\sum_G P_G = \id[A]$.

  There is also another, graphical, way of seeing this. Quantum measurements can also be characterized as `$B$-self-adjoint' coalgebras for the comonad $-\otimes B$, where being $B$-self-adjoint means that the following equation holds~\cite{coeckepavlovic:measurement}.
  \begin{equation}\label{eq:self-adjoint}
  \begin{pic}
   \node[morphism] (a)  at (0,0.75) {};
    \node[dot] (b) at (.5,0) {};
    \node[dot] (c) at (.5,-0.35) {};
   \draw ([xshift=-1mm]a.south west) to +(0,-1);
   \draw (a.north) to +(0,.4);
   \draw (b) to[out=0,in=-90] +(.3,.5) to +(0,.85);
   \draw ([xshift=1mm]a.south east) to[out=-90,in=180]  (b);
   \draw (b) to (c);
  \end{pic}
  \quad = \quad
  \begin{pic}
    \node[morphism,hflip] (f) at (-.4,0) {};
    \draw ([xshift=-1mm]f.north west) to +(0,.65);
    \draw (f.south) to +(0,-.65);
    \draw ([xshift=1mm]f.north east) to +(0,.65);
  \end{pic} 
  \end{equation}
  Such coalgebras correspond precisely to $\FEM$-algebras, as we will now show.
  Because of the dagger, coalgebras of the comonad $-\otimes B$ are just algebras of the monad $-\otimes B$. Thus it suffices to show that an algebra is FEM if and only if it satisfies~\eqref{eq:self-adjoint}. The implication from left to right is easy. 
	\[
	\begin{pic}[yscale=.7]
   		\node[morphism] (a)  at (0,0.75) {};
    	\node[dot] (b) at (.5,0) {};
    	\node[dot] (c) at (.5,-0.35) {};
   		\draw ([xshift=-1mm]a.south west) to +(0,-1);
   		\draw (a.north) to +(0,.4);
   		\draw (b) to[out=0,in=-90] +(.3,.5) to +(0,.85);
   		\draw ([xshift=1mm]a.south east) to[out=-90,in=180]  (b);
   		\draw (b) to (c);
  	\end{pic}
	\quad = \quad 
    \begin{pic}[yscale=.7]
   		\node[morphism,hflip] (a)  at (0,0) {};
    	\node[dot] (b) at (.5,0.75) {};
    	\node[dot] (c) at (.9,-0.5) {};
   		\draw ([xshift=-1mm]a.north west) to +(0,1);
   		\draw (a.south) to +(0,-.4);
   		\draw (b) to[out=0,in=90] +(.4,-.4) to (c);
   		\draw ([xshift=1mm]a.north east) to[out=90,in=180]  (b);
   		\draw (b) to +(0,.45);
    \end{pic}
    \quad = \quad
    \begin{pic}[yscale=.7]
    	\node[morphism,hflip] (f) at (-.4,0) {};
    	\draw ([xshift=-1mm]f.north west) to +(0,.65);
    	\draw (f.south) to +(0,-.65);
    	\draw ([xshift=1mm]f.north east) to +(0,.65);
    \end{pic} 
	\]
	The other implication can be proven as follows.
	\[
    \begin{pic}[yscale=.7]
   		\node[morphism,hflip] (a)  at (0,0) {};
    	\node[dot] (b) at (.5,0.75) {};
   		\draw ([xshift=-1mm]a.north west) to +(0,1.2);
   		\draw (a.south) to +(0,-.4);
   		\draw (b) to[out=0,in=90] +(.35,-.5) to +(0,-.85);
   		\draw ([xshift=1mm]a.north east) to[out=90,in=180]  (b);
   		\draw (b) to +(0,.64);
    \end{pic}
    \quad \overset{(\ref{eq:self-adjoint})}{=}  \quad
    \begin{pic}[yscale=.7]
   		\node[morphism] (a)  at (0,0.75) {};
    	\node[dot] (b) at (.5,0) {};
    	\node[dot] (c) at (.5,-0.35) {};
    	\node[dot] (d) at (1.1,0.75) {};
   		\draw ([xshift=-1mm]a.south west) to +(0,-1);
   		\draw (a.north) to +(0,.4);
   		\draw (b) to[out=0,in=180] (d);
   		\draw ([xshift=1mm]a.south east) to[out=-90,in=180]  (b);
   		\draw (d) to[out=0, in=90] +(0.3,-.5) to +(0,-.75);
   		\draw (d) to +(0,0.6);
   		\draw (b) to (c);
  	\end{pic}
  	\quad = \quad 
    \begin{pic}[yscale=.7]
   		\node[morphism] (a)  at (0,0.65) {};
    	\node[dot] (b) at (.3,0) {};
    	\node[dot] (c) at (.03,-0.6) {};
    	\node[dot] (d) at (.85,-0.6) {};
   		\draw ([xshift=-1mm]a.south west) to +(0,-1.5);
   		\draw (a.north) to +(0,.4);
   		\draw (b) to[out=0,in=180] (d);
   		\draw ([xshift=2.5mm]a.south east) to (b);
   		\draw (d) to[out=0, in=90] +(0.3,.5) to +(0,1.3);
   		\draw (d) to +(0,-0.5);
   		\draw (b) to[out=180, in=90] (c);
  	\end{pic}
  	\quad = \quad
    \begin{pic}[yscale=.7]
   		\node[morphism] (a)  at (0,0.75) {};
    	\node[dot] (b) at (.5,0) {};
   		\draw ([xshift=-1mm]a.south west) to +(0,-1);
   		\draw (a.north) to +(0,.4);
   		\draw (b) to[out=0,in=-90] +(.3,.5) to +(0,.85);
   		\draw ([xshift=1mm]a.south east) to[out=-90,in=180]  (b);
   		\draw (b) to +(0,-.45);
  	\end{pic}
   	\]
   	This finishes the alternative proof.
\end{proof}

\begin{example}\label{ex:femtensor}
  Let a groupoid $\cat{G}$ induce a dagger Frobenius monoid $B$ in $\cat{C}=\cat{Rel}$ or $\cat{C}=\cat{FHilb}$ as in Example~\ref{ex:groupoidFrob}. There is an equivalence $\FEM(-\otimes B) \simeq [\cat{G},\cat{C}]$.
\end{example}
\begin{proof}
  Separate the cases $\cat{Rel}$ and $\cat{FHilb}$.
  \begin{itemize}
  \item In $\cat{Rel}$, a FEM-algebra is a set $A$ with a \emph{relation} $g \colon A \to A$ for each $g \in B$ satisfying several equations.
    For each object $G$ of \cat{G}, define $A_G=\{a\in A\mid \id[G] a=\{a\} \}$. 
    The unit law implies that each $a \in A$ is in at least one $A_G$, and the other EM-law implies that no $a \in A$ can be in more than one $A_G$. 
    Now if $g\colon G\to H$ in \cat{G}, then $g$ defines a \emph{function} $A_G\to A_H$ and maps everything outside of $A_G$ to the empty set. 
    Thus the FEM-algebra $A$ defines an \emph{action} of \cat{G} in \cat{Rel}: a functor $F_A\colon\cat{G}\to\cat{Rel}$ making the sets $F_A(G)$ pairwise disjoint for distinct objects of \cat{G}. 

    Conversely, each such functor $F$ defines a FEM-algebra $A_F$ by setting $A_F=\bigcup_G F(G)$. 
    But the category of such functors is equivalent to $[\cat{G},\cat{Rel}]$.

  \item In $\cat{FHilb}$, a FEM-algebra is a Hilbert space $A$ with a morphism $g \colon A \to A$ for each $g \in B$. 
    For each object $G$ of \cat{G}, define $A_G=\{a\in A\mid \id[G] a=a \}$. 
    As above, $A$ is a direct sum of the $A_G$ and $g\colon G\to H$ in \cat{G} defines a morphism $A_G\to A_H$ and annihilates $A_G^\perp$. 
    This defines a representation of \cat{G} in \cat{FHilb}. 
    The Frobenius law implies that this representation is unitary.
  \end{itemize}
\end{proof}

In both of these examples, the same reasoning goes through over all of \cat{Hilb}. The fact that all of the categories of FEM-algebras from the previous example had daggers is no accident. 

\begin{lemma}\label{lem:femdag}
  Let $T$ be a dagger Frobenius monad on a dagger category $\cat{C}$. The dagger on $\cat{C}$ induces one on $\FEM(T)$.
\end{lemma} 
\begin{proof}
  Let $f \colon (A,a) \to (B,b)$ be a morphism of FEM-algebras; we are to show that $f^\dag$ is a morphism $(B,b)\to(A,a)$. 
  It suffices to show that $b \circ T(f) = f \circ a$ implies $a \circ T(f^\dag) = f^\dag \circ b$.
  Consider the following diagram:
  \[\begin{tikzpicture}[font=\small,scale=0.9]
    \matrix (m) [matrix of math nodes,row sep=3em,column sep=3em,minimum width=1em]
    {
      & T(B) & & T(A)  \\
     T^2(B) & T^2(B) & T^2(A) & T^2(A) & T(A)  \\
      &T(B) & &T(A) &       \\      
      &T(B) & &B &A    \\};
    \path[->]
    (m-1-2) edge node [above] {$Tf^\dag$} (m-1-4)
                  edge node [above] {$\mu^\dag$} (m-2-1)
                  edge node [left] {$Tb^\dag$}(m-2-2)
    (m-1-4)  edge node [above] {$\id$} (m-2-5)
                  edge node [left] {$\mu^\dag$}(m-2-4)
                  edge node [above] {$Ta^\dag$} (m-2-3) 
    (m-2-1) edge node [below] {$Tb$} (m-3-2)
                  edge node [below] {$\eta^\dag$} (m-4-2)
    (m-2-2) edge node [left] {$\mu$} (m-3-2) 
                 edge node [above] {$T^2f^\dag$} (m-2-3)
    (m-2-3) edge node [right] {$\mu$} (m-3-4) 
    (m-2-4) edge node [right] {$Ta$} (m-3-4) 
                 edge node [below] {$\eta^\dag$} (m-2-5)
    (m-2-5) edge node [right] {$a$} (m-4-5)
    (m-3-2) edge node [above] {$Tf^\dag$} (m-3-4)
    (m-3-4) edge node [above] {$\eta^\dag$} (m-4-5)
    (m-4-2) edge node [below] {$b$} (m-4-4) 
    (m-4-4) edge node [below] {$f^\dag$} (m-4-5);
    \node[gray] at (-4.25,1) {(i)};
    \node[gray] at (0,2) {(ii)};
    \node[gray] at (0,0) {(iii)};
    \node[gray] at (1.9,1) {(iv)};
    \node[gray] at (4.1,1.6) {(v)};
    \node[gray] at (4.7,-0.5) {(vi)};
    \node[gray] at (0,-2) {(vii)};
  \end{tikzpicture}\]
  Region (i) is the Frobenius law of $(B,b)$; commutativity of (ii) follows from the assumption that $f$ is a morphism $(A,a)\to (B,b)$ by applying $T$ and the dagger; (iii) is naturality of $\mu$; (iv) is the Frobenius law of $(A,a)$; (v) commutes since $T$ is a comonad; (vi) and (vii) commute by naturality of $\eta^\dag$. 
\end{proof}

In fact, the equivalence of Example~\ref{ex:femtensor} is a dagger equivalence.

\begin{lemma}\label{lem:femlaw} 
  Let $T$ be a dagger Frobenius monad. An EM-algebra $(A,a)$ is FEM if and only if $a^\dag$ is a homomorphism $(A,a)\to (TA,\mu_A)$.
\end{lemma}
\begin{proof} 
  If $(A,a)$ is a FEM-algebra, its associativity means that $a$ is a homomorphism $(TA,\mu_A) \to (A,a)$. Here $T(A)$ is a FEM-algebra too because $T$ is a dagger Frobenius monad. Thus $a^\dag$ is a homomorphism $(A,a) \to (TA,\mu_A)$ by Lemma~\ref{lem:femdag}.

  For the converse, assume $a^\dag$ is a homomorphism $(A,a)\to (TA,\mu_A)$, so the diagram
  \[
  \begin{tikzpicture}
     \matrix (m) [matrix of math nodes,row sep=2em,column sep=4em,minimum width=2em]
     {
      T(A) & T^2(A) \\
      A & T(A) \\};
     \path[->]
     (m-1-1) edge node [left] {$a$} (m-2-1)
             edge node [above] {$Ta^\dag$} (m-1-2)
     (m-2-1) edge node [below] {$a^\dag$} (m-2-2)
     (m-1-2) edge node [right] {$\mu $} (m-2-2);
  \end{tikzpicture}
  \]
  commutes. 
  Hence $\mu \circ Ta^\dag$ is self-adjoint as $a^\dag\circ a$ is, giving the Frobenius law~\eqref{eq:femlaw}.
\end{proof}

Interpreting the associative law for algebras as saying that $a\colon TA\to A$ is a homomorphism $(TA,\mu_A)\to (A,a)$, Lemmas~\ref{lem:femdag} and~\ref{lem:femlaw} show that this morphism is universal, in the sense that if its dagger is an algebra homomorphism, then so is the dagger of any other algebra homomorphism to $A$ (whose domain satisfies the Frobenius law).

\begin{theorem}\label{thm:comparison}
  Let $F$ and $G$ be dagger adjoints, and write $T=G\circ F$ for the induced dagger Frobenius monad. 
  There are unique dagger functors $K$ and $J$ making the following diagram commute.
  \[\begin{tikzpicture}[xscale=4,yscale=2]
  	\node (tl) at (-1,1) {$\Kl(T)$};
  	\node (t) at (0,1) {$\cat D$};
  	\node (tr) at (1,1) {$\FEM(T)$};
  	\node (b) at (0,0) {$\cat C$};
  	\draw[->, dashed] (tl) to node[above] {$K$} (t);
  	\draw[->, dashed] (t) to node[above] {$J$} (tr);
  	\draw[->] ([xshift=.2mm]t.south) to node[right] {$G$} ([xshift=.2mm]b.north);
  	\draw[->] ([xshift=-.2mm]b.north) to node[left] {$F$} ([xshift=-.2mm]t.south);
  	\draw[->] (tl.-45) to (b.135);
  	\draw[->] ([xshift=-.5mm,yshift=-.5mm]b.135) to ([yshift=-.5mm,xshift=-.5mm]tl.-45);
  	\draw[->] ([xshift=.5mm,yshift=-.5mm]tr.-135) to ([yshift=-.5mm,xshift=.5mm]b.45);
  	\draw[->] (b.45) to (tr.-135);
  \end{tikzpicture}\]
  Moreover, $J$ is full, $K$ is full and faithful, and $J\circ K$ is the canonical inclusion.
\end{theorem}
\begin{proof}
  It suffices to show that the comparisons $K \colon \Kl(T) \to \cat{D}$ and $J \colon \cat{D} \to \EM(T)$ (see~\cite[VI.3, IV.5]{maclane:categories}) are dagger functors, and that $J$ factors through $\FEM(T)$.

  Let $(A,a)$ be in the image of $J$. As $J\circ K$ equals the canonical inclusion, $(T(A),\mu_A)$ is also in the image of $J$. Because $J$ is full, the homomorphism $a\colon (TA,\mu_A)\to (A,a)$ is in the image of $J$, say $a=J(f)$. But then $G(f)=a$, and so $G(f^\dag)=a^\dag$. This implies that $a^\dag$ is a homomorphism $(A,a)\to (TA,\mu_A)$. Lemma~\ref{lem:femlaw} guarantees $(A,a)$ is in $\FEM(T)$. 
 
  Clearly $J$ is a dagger functor. It remains to show that $K$ is a dagger functor. As $K$ is full and faithful, Lemma~\ref{lem:daglem} gives $\Kl(T)$ a unique dagger making it a dagger functor. This also makes $J\circ K$ a dagger functor, and since $J\circ K$ equals the canonical inclusion, the induced dagger on $\Kl(T)$ must equal the canonical one from Lemma~\ref{lem:kleislidagger}.
\end{proof}

The previous theorem leads to a direct proof of Lemma~\ref{lem:daggeradjunction} above, as follows. The definition of $\FEM(T)$ makes sense for arbitrary monads (that might not satisfy the Frobenius law), and the proof above still shows that the image of $J \colon \cat{D} \to \EM(T)$ lands in $\FEM(T)$. Hence every free algebra is a FEM-algebra. This implies the Frobenius law~\eqref{eq:frobeniusformonads} for $T$.

\section{Formal monads on dagger categories}\label{sec:formal}

Both ordinary monads~\cite{street1972formal} and Frobenius monads~\cite{lauda:ambidextrous} have been treated formally in 2-categories. This section establishes the counterpart for dagger 2-categories. Its main contribution is to show that the category of FEM-algebras satisfies a similar universal property for dagger Frobenius monads as EM-algebras do for ordinary monads. 
Recall from Definition~\ref{def:dagtwocat} that a dagger 2-category is a category enriched in $\cat{DagCat}$.

\begin{definition} 
  An adjunction in a dagger 2-category is just an adjunction in the underlying 2-category. 
  A \emph{dagger 2-adjunction} consists of two $\cat{DagCat}$-enriched functors that form a 2-adjunction in the usual sense. 
\end{definition}

Adjunctions in a dagger 2-category need not specify left and right, just like the dagger adjunctions they generalize, but dagger 2-adjunctions need to specify left and right. 

\begin{definition} 
  Let \cat{C} be a dagger 2-category. 
  A \emph{dagger Frobenius monad} consists of an object $C$, a morphism $T\colon C\to C$, and 2-cells $\eta\colon \id[C]\to T$ and $\mu\colon T^2\to T$ that form a monad in the underlying 2-category of $\cat{C}$ and satisfy the Frobenius law
  \[
      \mu T\circ T\mu^\dag=T\mu\circ \mu^\dag T.
  \]
  A \emph{morphism of dagger Frobenius monads} $(C,S)\to (D,T)$ is a morphism $F\colon C\to D$ with a 2-cell $\sigma\colon TF\to FS$ making the following diagrams commute.
  \begin{equation}\label{diag:mapofFrobmonads}
    \begin{aligned}\begin{tikzpicture}
     \matrix (m) [matrix of math nodes,row sep=1em,column sep=2em]
     {
       & TF \\
      F \\
       & FS \\};
     \tikzset{font=\small};
     \path[->]
     (m-2-1) edge node [below=1ex,left] {$F\eta^S$} (m-3-2)
             edge node [above=1ex,left] {$\eta^TF$} (m-1-2)
     (m-1-2) edge node [left] {$\sigma$} (m-3-2);
    \end{tikzpicture}\end{aligned}
    \begin{aligned}\begin{tikzpicture}
     \matrix (m) [matrix of math nodes,row sep=1em,column sep=2em]
     {
       & TFS & FSS \\
      TTF & & \\
       & TF & FS \\};
     \tikzset{font=\small};
     \path[->]
     (m-2-1) edge node [below=1ex,left] {$\mu^T F$} (m-3-2)
             edge node [above=1ex,left] {$T\sigma$} (m-1-2)
     (m-1-2) edge node [above] {$\sigma S$} (m-1-3)
     (m-1-3) edge node [left] {$F\mu^S$} (m-3-3)
     (m-3-2) edge node [below] {$\sigma$} (m-3-3);
    \end{tikzpicture}\end{aligned}
    \begin{aligned}\begin{tikzpicture}
     \matrix (m) [matrix of math nodes,row sep=1em,column sep=2em]
     {
       & FSS & FS \\
      TFS & & \\
       & TTF & TF \\};
     \tikzset{font=\small};
     \path[->]
     (m-2-1) edge node [below=1ex,left] {$T\sigma^\dag$} (m-3-2)
             edge node [above=1ex,left] {$\sigma S$} (m-1-2)
     (m-1-2) edge node [above] {$F\mu^S$} (m-1-3)
     (m-1-3) edge node [left] {$\sigma^\dag$} (m-3-3)
     (m-3-2) edge node [below] {$\mu^T F$} (m-3-3);
    \end{tikzpicture}\end{aligned}
  \end{equation}
  A \emph{transformation between morphisms of dagger Frobenius monads} $(F,\sigma) \to (G,\tau)$ is a 2-cell $\phi\colon F\to G$ making the following diagrams commute.
  \[
    \begin{tikzpicture}
     \matrix (m) [matrix of math nodes,row sep=2em,column sep=4em,minimum width=2em]
     {
      TF & TG & \\
      FS & GS \\};
     \path[->]
     (m-1-1) edge node [left] {$\sigma$} (m-2-1)
             edge node [above] {$T\phi$} (m-1-2)
     (m-2-1) edge node [below] {$\phi S$} (m-2-2)
     (m-1-2) edge node [right] {$\tau$} (m-2-2);
    \end{tikzpicture}
  \quad
    \begin{tikzpicture}
     \matrix (m) [matrix of math nodes,row sep=2em,column sep=4em,minimum width=2em]
     {
      TG & TF & \\
      GS & FS \\};
     \path[->]
     (m-1-1) edge node [left] {$\sigma$} (m-2-1)
             edge node [above] {$T\phi^\dag$} (m-1-2)
     (m-2-1) edge node [below] {$\phi ^\dag S$} (m-2-2)
     (m-1-2) edge node [right] {$\tau$} (m-2-2);
    \end{tikzpicture}
  \]
  Define the composition of morphisms to be $(G,\tau) \circ (F,\sigma) = (GF,G\sigma \circ \tau F)$.
  Horizontal and vertical composition of 2-cells in $\cat{C}$ define horizontal and vertical composition of transformations of morphisms of dagger Frobenius monads, and the dagger on 2-cells of $\cat{C}$ gives a dagger on these transformations. 
  This forms a dagger 2-category $\DFMonad(\cat{C})$.
\end{definition}

Omitting the third diagram of~\eqref{diag:mapofFrobmonads} gives the usual definition of a monad morphism.
We require this coherence with the dagger for the following reason: just as the first two diagrams of~\eqref{diag:mapofFrobmonads} ensure that $(C,T)\mapsto \EM(T)$ is a 2-functor $\mathrm{Monad}(\cat{Cat}) \to \cat{Cat}$, the third one ensures that $(C,T) \mapsto \FEM(T)$ is a dagger 2-functor $\DFMonad(\cat{DagCat}) \to \cat{DagCat}$.

There is an inclusion dagger 2-functor $\cat{C}\to\DFMonad(\cat{C})$ given by $C\mapsto (C,\id[C])$, $F\mapsto (F,\id)$, and $\psi\mapsto \psi$. There is also a forgetful dagger 2-functor $\DFMonad(\cat{C})\to \cat{C}$ given by $(C,T)\mapsto C$, $(F,\sigma)\mapsto F$, and $\psi\mapsto\psi$. As with the formal theory of monads~\cite{street1972formal}, the forgetful functor is the left adjoint to the inclusion functor. To see this, it suffices to exhibit a natural isomorphism of dagger categories $[C,D]\cong [(C,T),(D,\id[D])]$, such as  sending $F\colon C\to D$ to $(F,F\eta^T)$ and $\psi$ to $\psi$. 

\begin{definition}
  A dagger 2-category \cat{C} \emph{admits the construction of FEM-algebras} when the inclusion $\cat{C}\to\DFMonad(\cat{C})$ has a right adjoint $\FEM \colon \DFMonad(\cat{C}) \to \cat{C}$. 
\end{definition}

We will abbreviate $\FEM(C,T)$ to $\FEM(T)$ when no confusion can arise.

\begin{theorem} 
  If \cat{C} admits the construction of FEM-algebras, then dagger Frobenius monads factor as dagger adjunctions.
\end{theorem}
\begin{proof} 
  We closely follow the proof of the similar theorem for ordinary monads in~\cite{street1972formal}, but need to verify commutativity of some additional diagrams. 
  To verify that $(T,\mu)$ is a morphism of dagger Frobenius monads $(C,\id)\to (C,T)$, the first two diagrams are exactly as for ordinary monads, and the third diagram commutes by Lemma~\ref{lem:extendedfrobeniuslaw}.
  Denoting the counit of the adjunction of the assumption by $(E,\varepsilon)\colon (\FEM(T),\id)\to (C,T)$, the universal property gives a unique morphism $(J,\id)\colon (C,\id)\to (\FEM(T),\id)$ making the following diagram commute.
  \[
    			\begin{tikzpicture}
			     \matrix (m) [matrix of math nodes,row sep=2em,column sep=4em,minimum width=2em]
			     {
			       (C,\id)&  & (C,T) \\
			       & (\FEM(T),\id ) \\};
			     \path[->]
			     (m-1-1) edge node [below=1ex,left] {$(J,\id)$} (m-2-2)
			             edge node [above] {$(T,\mu)$} (m-1-3)
			     (m-2-2) edge node [below=1ex,right] {$(E,\varepsilon)$} (m-1-3);
			    \end{tikzpicture}
  \]
  Thus $T=EJ$ and $\mu=\varepsilon J$. Next we verify $\varepsilon$ is a transformation of morphisms of dagger Frobenius monads $(EJE,\mu E)\to (E,\varepsilon)$, by showing the following diagrams commute.
  \[
		    \begin{tikzpicture}
		     \matrix (m) [matrix of math nodes,row sep=2em,column sep=4em,minimum width=2em]
		     {
		      TEJE & TE \\
		      TE & E \\};
		     \path[->]
		     (m-1-1) edge node [left] {$\mu E$} (m-2-1)
		             edge node [above] {$T\varepsilon$} (m-1-2)
		     (m-2-1) edge node [below] {$\varepsilon$} (m-2-2)
		     (m-1-2) edge node [right] {$\varepsilon$} (m-2-2);
		    \end{tikzpicture}
		  \qquad
		    \begin{tikzpicture}
		     \matrix (m) [matrix of math nodes,row sep=2em,column sep=4em,minimum width=2em]
		     {
		      TE & TEJE \\
		      E & TE \\};
		     \path[->]
		     (m-1-1) edge node [left] {$\varepsilon$} (m-2-1)
		             edge node [above] {$T\varepsilon^\dag$} (m-1-2)
		     (m-2-1) edge node [below] {$\varepsilon^\dag$} (m-2-2)
		     (m-1-2) edge node [right] {$\mu E$} (m-2-2);
		    \end{tikzpicture}
  \]
  These diagrams are instances of the second and third equations of~\eqref{diag:mapofFrobmonads}.

  The rest of the proof proceeds exactly as in~\cite{street1972formal}. As $\varepsilon$ is a transformation, the adjunction gives a unique 2-cell $\xi\colon JE\to \id[\FEM(T)]$ in \cat{C} with $E\xi=\varepsilon$. Now 
  \begin{align*}
 	E(\xi J\circ J\eta)
 	=E\xi J\circ EJ\eta 
	=\varepsilon J\circ T\eta 
	=\mu\circ T\eta=\id,
  \end{align*}
  and so $\xi J \circ J\eta =\id$ by the universal property of the counit. 
  Furthermore
  \[
		 	E\xi \circ \eta E=\varepsilon \circ \eta E=\id,
  \] 
  so $(E,J,\mu,\xi)$ is an adjunction generating $(C,T)$.
\end{proof}

\begin{theorem} 
  The dagger 2-category \cat{DagCat} admits the construction of FEM-algebras.
\end{theorem} 
\begin{proof}
  We first show $(\cat{C},T)\mapsto (\FEM(T))$ extends to a well-defined dagger 2-functor $\FEM \colon \DFMonad(\cat{DagCat})\to\cat{DagCat}$. On a 1-cell $(F,\sigma)\colon(\cat{C},S)\to (\cat{D},T)$, define a dagger functor $\FEM(F,\sigma)\colon \FEM(S)\to \FEM(T)$ as follows: map the algebra $a\colon SA\to A$ to $F(a)\circ \sigma_A\colon TFA\to FSA\to FA$ and the homomorphism $f\colon (A,a)\to (B,b)$ to $F(f)$. The laws for FEM-algebras of $(F,\sigma)$ now show that $\FEM(F,\sigma)$ sends FEM-algebras to FEM-algebras. 
  On a 2-cell $\psi\colon (F,\sigma)\to (G,\tau)$, define $\FEM(\psi)$ by $\FEM(\psi)_A=\psi_A$; each $\psi_A$ is a morphism of FEM-algebras.

  There is a natural isomorphism of dagger categories $[\cat{C},\FEM(T)]\cong [(\cat{C},\id[\cat{C}]),(\cat{D},T)]$ as follows. 
  A dagger functor $\cat{C}\to \FEM(T)$ consists of a dagger functor $F\cat{C}\to \cat{D}$ and a family of maps $\sigma_A\colon TFA\to FA$ making each $(FA,\sigma_A)$ into a FEM-algebra. 
  Then $(F,\sigma)\colon(\cat{C},\id[\cat{C}])\to (\cat{D},T)$ is a well-defined morphism of dagger Frobenius monads. Similarly, any $(F,\sigma)\colon(\cat{C},\id[\cat{C}])\to (\cat{D},T)$ defines a dagger functor $\cat{C}\to \FEM(T)$. On the level of natural transformations both of these operations are obvious. 
\end{proof}

In~\cite{street1972formal} this last result is proved as follows: instead of starting with the definition of $\cat{C}^T$, the construction is recovered from the assumption that the right adjoint exists and considering functors from the categories \cat{1} and \cat{2} to the right adjoint, thus recovering the objects and arrows of $\cat{C}^T$. It is unclear how to write a similar proof in our case: while $[\cat{2},\cat{C}]$ classifies arrows and commutative squares for an ordinary category, it is not obvious how to replace \cat{2} with a dagger category playing an analogous role.

\section{Strength}\label{sec:strength}

As we saw in Example~\ref{example:tensor}, (Frobenius) monoids in a monoidal dagger category induce (dagger Frobenius) monads on the category. This in fact sets up an adjunction between monoids and strong monads~\cite{wolff:monads}. This section shows that the Frobenius law promotes this adjunction to an equivalence. Most of this section generalizes to the non-dagger setting.

\begin{definition}\label{def:strength}
  A dagger functor $F$ between monoidal dagger categories is \emph{strong} if it is equipped with natural unitary morphisms
  $\st_{A,B} \colon A \otimes F(B) \to F(A\otimes B)$ satisfying $\st \circ \alpha = F(\alpha) \circ \st \circ (\id \otimes \st)$ and $F(\lambda) \circ \st = \lambda$.
  A (dagger Frobenius) monad on a monoidal dagger category is \emph{strong} if it is a strong dagger functor with $\st \circ (\id \otimes \mu) = \mu \circ T(\st) \circ \st$ and $\st \circ (\id \otimes \eta) = \eta$.
  A morphism $\beta$ of dagger Frobenius monads is \emph{strong} if $\beta \circ \st = \st \circ (\id \otimes \beta)$.
\end{definition} 

To prove the equivalence between dagger Frobenius monoids and strong dagger Frobenius monads, we need two lemmas.
 
\begin{lemma}\label{lem:froblaw} 
  If $T$ is a strong dagger Frobenius monad on a monoidal dagger category, then $T(I)$ is a dagger Frobenius monoid.
\end{lemma}
\begin{proof}
  Consider the diagram in \figurename~\ref{fig:froblaw}. 
  Region (i) commutes because $T$ is a dagger Frobenius monad, (ii) because $\mu^\dag$ is natural, (iii) because $\rho^{-1}$ is natural, (iv) because $\st^\dag$ is natural, (v) is a consequence of $T$ being a strong monad, (vi) commutes as $\rho$ is natural, (vii) and (viii) because $\st$ is natural, (ix) commutes trivially and (x) because $\st$ is natural. Regions (ii)'-(x)' commute for dual reasons. 
  Hence the outer diagram commutes.
\end{proof}

\begin{figure*}
  \centering
  \begin{sideways}
  \begin{tikzpicture}[font=\footnotesize,xscale=.825,yscale=.825]
  \matrix (m) [matrix of math nodes,row sep=2em,column sep=1.75em]
  {
     T(I)\otimes T(T(I)\otimes I) & T(I)\otimes (T(I)\otimes T(I)) & & (T(I)\otimes T(I))\otimes T(I) & T(T(I)\otimes I)\otimes T(I)) \\
     &T(T(I)\otimes (T(I)\otimes I)) & T((T(I)\otimes T(I))\otimes I) & T(T(I)\otimes I)\otimes T(I) & \\
     T(I)\otimes T^2(I) & T(T(I)\otimes T(I)) & T(T(T(I)\otimes I)\otimes I) & T(T^2(I)\otimes I) & T^2(I)\otimes T(I) \\
     &T^2(T(I)\otimes I) &T^3(I) & T(T(I)\otimes I)& \\ 
     T(I)\otimes T(I) & T^2(I) & &T^2(I) & T(I)\otimes T(I) \\
     & T(T(I)\otimes I) & T^3(I) & T^2(T(I)\otimes I) & \\
     T^2(I)\otimes T(I) & T(T^2(I)\otimes I) & T(T(T(I)\otimes I)\otimes I) & T(T(I)\otimes T(I)) & T(I)\otimes T^2(I) \\
     & T(T(I)\otimes I)\otimes T(I) & T((T(I)\otimes T(I))\otimes I) & T(T(I)\otimes (T(I)\otimes I)) & \\
     T(T(I)\otimes I)\otimes T(I) & (T(I)\otimes T(I))\otimes T(I) & & T(I)\otimes (T(I)\otimes T(I)) & T(I)\otimes T(T(I)\otimes I) \\};
   \path[->]
    (m-1-1) edge node [above] {$\id\otimes \st^\dag$} (m-1-2)
    (m-1-2) edge node [above] {$\alpha$} (m-1-4)
    (m-1-4) edge node [above] {$\st\otimes\id$} (m-1-5)
    (m-1-5) edge node [right] {$T(\rho)\otimes\id$} (m-3-5)
    (m-1-1) edge node [above] {$\st$} (m-2-2)
    (m-2-2) edge node [above] {$T(\alpha)$} (m-2-3)
    (m-2-3) edge node [above] {$\st^\dag$} (m-1-4)
    (m-2-4) edge node [right] {$\st^\dag\otimes\id$} (m-1-4) 
    (m-2-4) edge node [above] {$\quad\quad T(\rho)\otimes \id$} (m-3-5)
    (m-3-1) edge node [above] {$\st$} (m-3-2)
    (m-3-1) edge node [left] {$\id\otimes T(\rho^{-1})$} (m-1-1)
    (m-3-2) edge node [left] {$T(\id\otimes\rho^{-1})$} (m-2-2)
    (m-3-2) edge node [below] {$\quad T(\rho^{-1})$} (m-2-3)
    (m-3-3) edge node [above] {$T(T(\rho)\otimes\id)$} (m-3-4)
    (m-3-3) edge node [right] {$T(\st^\dag\otimes\id)$} (m-2-3)
    (m-3-3) edge node [right] {$\st^\dag$} (m-2-4)
    (m-3-4) edge node [above] {$\st^\dag$} (m-3-5)
    (m-3-4) edge node [right] {$T(\mu\otimes\id)$} (m-4-4)
    (m-3-5) edge node [right] {$\mu\otimes\id$} (m-5-5)
    (m-4-2) edge node [left] {$T(\st^\dag)$} (m-3-2)
    (m-4-2) edge node [right] {$\quad T(\rho^{-1})$} (m-3-3)
    (m-4-2) edge node [above] {$T^2(\rho)$} (m-4-3)
    (m-4-3) edge node [above] {$T(\mu_I)$} (m-5-4)
    (m-4-4) edge node [above] {$\st^\dag$} (m-5-5)
    (m-5-1) edge node [above] {$\st$} (m-6-2)
    (m-5-1) edge node [left] {$\id\otimes\mu^\dag$} (m-3-1)
    (m-5-1) edge node [left] {$\mu^\dag\otimes\id$} (m-7-1)
    (m-5-2) edge node [above] {$\mu^\dag_{T(I)}$} (m-4-3)
    (m-5-2) edge node [above] {$T(\mu^\dag_I)$} (m-6-3)
    (m-5-4) edge node [left] {$T(\rho^{-1})$} (m-4-4)
    (m-6-2) edge node [right] {$T(\rho)$} (m-5-2)
    (m-6-2) edge node [left] {$T(\mu^\dag\otimes\id)$} (m-7-2)
    (m-6-3) edge node [above] {$\mu_{T(I)}$} (m-5-4)
    (m-6-3) edge node [above] {$\quad T^2(\rho^{-1})$} (m-6-4)
    (m-7-1) edge node [above] {$\st$} (m-7-2)
    (m-7-1) edge node [left] {$T(\rho^{-1})\otimes\id$} (m-9-1)
    (m-7-1) edge node [above] {$\ \ \quad T(\rho^{-1})\otimes\id$} (m-8-2)
    (m-7-2) edge node [above] {$T(T(\rho^{-1})\otimes\id)$} (m-7-3)
    (m-7-3) edge node [above] {$T(\rho)$} (m-6-4)
    (m-7-4) edge node [right] {$T(\st)$} (m-6-4)
    (m-7-4) edge node [above] {$\st^\dag$} (m-7-5)
    (m-7-5) edge node [right] {$\id\otimes\mu$} (m-5-5)
    (m-8-2) edge node [above] {$\st$} (m-7-3)
    (m-8-3) edge node [right] {$T(\st\otimes\id)$} (m-7-3)
    (m-8-3) edge node [right] {$\quad T(\rho)$} (m-7-4)
    (m-8-3) edge node [above] {$T(\alpha^{-1})$} (m-8-4)
    (m-8-4) edge node [right] {$T(\id\otimes\rho)$} (m-7-4)
    (m-8-4) edge node [above] {$\ \ \st^\dag$} (m-9-5)
    (m-9-1) edge node [below] {$\st^\dag\otimes\id $} (m-9-2)
    (m-9-2) edge node [left] {$\st\otimes\id $} (m-8-2)
    (m-9-2) edge node [above] {$\st$} (m-8-3)
    (m-9-2) edge node [below] {$\alpha^{-1}$} (m-9-4)
    (m-9-4) edge node [below] {$\id\otimes \st$} (m-9-5)
    (m-9-5) edge node [right] {$\id\otimes T(\rho)$} (m-7-5)
    (m-6-2) edge [bend left=30] (m-4-2)
    (m-6-4) edge [bend right=30] (m-4-4);
    \node at (-6.5,0) {$\mu^\dag$};
    \node at (6.25,0) {$\mu$}; 
    \node[gray] at (0,0) {(i)}; 
    \node[gray] at (-5,1) {(ii)}; 
    \node[gray] at (5,-1) {(ii)'}; 
    \node[gray] at (3,2) {(iii)}; 
    \node[gray] at (-3,-2) {(iii)'};
    \node[gray] at (8,2.5) {(iv)}; 
    \node[gray] at (-8,-2.5) {(iv)'}; 
    \node[gray] at (-8,1.5) {(v)'}; 
    \node[gray] at (8,-1.5) {(v)}; 
    \node[gray] at (2.75,-3.5) {(vi)}; 
    \node[gray] at (-2.75,3.5) {(vi)'};
    \node[gray] at (-2.75,-5.25) {(vii)}; 
    \node[gray] at (2.75,5.25) {(vii)'}; 
    \node[gray] at (-5,-4.5) {(viii)}; 
    \node[gray] at (5,4.5) {(viii)'};
    \node[gray] at (-8,-6) {(ix)}; 
    \node[gray] at (8,6) {(ix)'};
    \node[gray] at (-8,5) {(x)};
    \node[gray] at (8,-5) {(x)'}; 
    \node[gray] at (-4.25,4.75) {(xi)}; 
    \node[gray] at (4.5,-4.75) {(xi)'};
    \node[gray] at (-2.65,6.5) {(xii)}; 
    \node[gray] at (2.65,-6.5) {(xii)'};
  \end{tikzpicture}
  \end{sideways}
  \caption{Diagram proving that $T\mapsto T(I)$ preserves the Frobenius law.}
  \label{fig:froblaw}
\end{figure*}

\begin{lemma}\label{lem:strictmorphism}
  If $T$ is a strong dagger Frobenius monad on a monoidal dagger category, then $T\rho \circ \st \colon A \otimes T(I) \to T(A)$ preserves $\eta^\dag$ and $\mu^\dag$.
\end{lemma}
\begin{proof}
  To show that $\eta^\dag$ is preserved, it suffices to see that
  \[\begin{tikzpicture}[font=\small]
  \matrix (m) [matrix of math nodes,row sep=2em,column sep=4em,minimum width=2em]
  {
     A\otimes T(I) & \\
     T(A\otimes I) & A\otimes I \\ 
     T(A) &  A\\};
  \path[-stealth]
    (m-1-1) edge node [left] {$\st_{A,I}$} (m-2-1)
            edge node [right] {$\ \ \id\otimes \eta^\dag_I$} (m-2-2)
    (m-2-1) edge node [below] {$\eta^\dag_{A\otimes I}$} (m-2-2)
    (m-2-1) edge node [left] {$T(\rho_A)$} (m-3-1)
    (m-2-2) edge node [right] {$\rho_A$} (m-3-2)
     (m-3-1) edge node [below] {$\eta^\dag_A$} (m-3-2);
  \end{tikzpicture}\]
  commutes. But the rectangle commutes because $\eta^\dag$ is natural, and the triangle commutes because $T$ is a strong monad and strength is unitary.

  To see that $\mu^\dag$ is preserved, consider the following diagram:
  \[\begin{tikzpicture}[font=\small,yscale=.75,xscale=.8]
    \matrix (m) [matrix of math nodes,row sep=2em,column sep=2.5em]
    {
     A\otimes T(I) & A\otimes T^2(I) &A\otimes T(T(I)\otimes I) &A\otimes (T(I)\otimes T(I)) \\
      & A\otimes T^2(I) &A\otimes T(T(I)\otimes I) &(A\otimes T(I))\otimes T(I) \\
      &  &T(A\otimes (T(I)\otimes I)) &T((A\otimes T(I))\otimes I) \\
      &  & &T(A\otimes T(I)) \\
     T(A\otimes I) &  & &T^2(A\otimes I) \\
     T(A) &  & &T^2(A) \\};
    \path[->]
    (m-1-1) edge node [left] {$\st$} (m-5-1)
            edge node [above] {$\id\otimes\mu^\dag$} (m-1-2)
    (m-1-2) edge node [above] {$\id\otimes T(\rho^{-1})$} (m-1-3)
    (m-1-3) edge node [above] {$\id\otimes \st^\dag$} (m-1-4)
    (m-2-3) edge node [above] {$\id\otimes T(\rho)$} (m-2-2)
    (m-1-2) edge node [right] {$\id$} (m-2-2)
    (m-1-4) edge node [right] {$\alpha$} (m-2-4)
    (m-2-4) edge node [right] {$\st$} (m-3-4)
    (m-3-4) edge node [right] {$T(\rho)$} (m-4-4)
    (m-4-4) edge node [right] {$T(\st)$} (m-5-4)
    (m-5-4) edge node [right] {$T^2(\rho)$} (m-6-4)
    (m-5-1) edge node [left] {$T(\rho)$} (m-6-1) 
    (m-6-1) edge node [below] {$\mu^\dag$} (m-6-4)
    (m-2-3) edge node [right] {$\st$} (m-3-3) 
    (m-1-4) edge node [below] {$\quad\id\otimes \st$} (m-2-3) 
    (m-3-3) edge node [above] {$T(\alpha)$} (m-3-4)
    (m-3-3) edge node [below] {$T(\id\otimes\rho)\quad\quad$} (m-4-4)
    (m-5-1) edge node [below] {$\mu^\dag$} (m-5-4)
    (m-2-2) edge [bend right=30] (m-4-4);
    \node at (-.2,-0.7) {$\st$};
    \node[gray] at (-4,1) {(i)};
    \node[gray] at (0,4) {(ii)};
    \node[gray] at (0,2) {(iii)};
    \node[gray] at (4,2) {(iv)};
    \node[gray] at (4,0.5) {(v)}; 
    \node[gray] at (0,-4) {(vi)};
  \end{tikzpicture}\]
  Commutativity of region (i) is a consequence of strength being unitary, (ii) commutes by definition, (iii) commutes as strength is natural, (iv) because $T$ is a strong functor, (v) by coherence and finally (vi) by naturality of $\mu^\dag$. 
\end{proof}

\begin{theorem}\label{thm:strong}
  Let $\cat{C}$ be a monoidal dagger category. The operations $B \mapsto - \otimes B$ and $T \mapsto T(I)$ form an equivalence between dagger Frobenius monoids in $\cat{C}$ and strong dagger Frobenius monads on $\cat{C}$.
\end{theorem} 
\begin{proof}
  It is well-known that $B \mapsto - \otimes B$ is left adjoint to $T \mapsto T(I)$, when considered as maps between ordinary monoids and ordinary strong monads, see~\cite{wolff:monads}; the unit of the adjunction is $\lambda \colon I \otimes B \to B$, and the counit is determined by $T\rho \circ \st \colon A \otimes T(I)\to T(A)$. 

  Example~\ref{example:tensor} already showed that $B\mapsto -\otimes B$ preserves the Frobenius law. 
  Lemma~\ref{lem:froblaw} shows that $T \mapsto T(I)$ preserves the Frobenius law, too.
  It remains to prove that they form an equivalence. 
  Clearly the unit of the adjunction is a natural isomorphism. 
  To see that the counit is also a natural isomorphism, combine Lemmas~\ref{lem:strictmorphism} and~\ref{lem:strictmorphismsareiso}.
\end{proof}

It follows from the previous theorem that not every dagger Frobenius monad is strong: as discussed in Section~\ref{sec:monads}, the monad induced by the dagger adjunction of Example~\ref{ex:imdagadj} is not of the form $- \otimes B$ for fixed $B$.

\begin{remark}\label{rem:ratherstrong}
  One might think it too strong to require strength to be unitary. But without it, Theorem~\ref{thm:strong} no longer holds.
\end{remark}
\begin{proof}
  Let us temporarily call a dagger Frobenius monad \emph{rather strong} when it is simultaneously a strong monad. The operations of Theorem~\ref{thm:strong} do not form an adjunction between Frobenius monoids and rather strong Frobenius monads, because the counit of the adjunction would not be a well-defined morphism. Producing a counterexample where the counit does not preserve $\mu^\dag$ comes down to finding a rather strong Frobenius monad with $T(\eta_A) \circ \eta_A \neq \mu_A^\dag \circ \eta_A$ for some $A$. This is the case when $T = - \otimes B$ for a dagger Frobenius monoid $B$ with $\tinyunit \otimes \tinyunit \neq (\tinymult)^\dag \circ \tinyunit$. Such dagger Frobenius monoids certainly exist: if $G$ is any nontrivial group, regarded as a dagger Frobenius monoid in $\cat{Rel}$ as in Example~\ref{ex:groupoidFrob}, then $\tinyunit \otimes \tinyunit$ is the relation $\{(*,(1,1))\}$, but $(\tinymult)^\dag \circ \tinyunit = \{(*,(g,g^{-1})) \mid g \in G\}$.
\end{proof} 

Conceivably, there might be a notion of strength that is weaker than Definition~\ref{def:strength} but stronger than Remark~\ref{rem:ratherstrong}, that would still have reasonable properties. It seems that one would want at least Lemma~\ref{lem:froblaw} to go through, but the proof we've given seems to use invertibility of the strength maps in an essential way. In any case, one wants the underlying monad and comonad to be both strong and costrong. The monad from Example~\ref{ex:imdagadj} satisfies Remark~\ref{rem:ratherstrong} under $\st=\eta\otimes \id$, but then fails to be costrong as a monad.

A Frobenius monoid in a monoidal dagger category is \emph{special} when 
$\smash{\tinymult \circ\tinycomult=\begin{pic} \draw (0,0) to (0,.45); \end{pic}}$.
Theorem~\ref{thm:strong} restricts to an equivalence between special dagger Frobenius monoids and special strong dagger Frobenius monads.

For symmetric monoidal dagger categories, there is also a notion of commutativity for strong monads~\cite{kock:strong,jacobs:weakening}. 
Given a strong dagger Frobenius monad $T$, one can define a natural transformation $\st'_{A,B} \colon T(A) \otimes B \to T(A \otimes B)$ by $T(\sigma_{B,A})\circ st_{B,A}\circ \sigma_{T(A),B}$, and
\begin{align*} 
  dst_{A,B} & = \mu_{A\otimes B}\circ T(st'_{A,B})\circ st_{T(A),B}, \\ 
  dst_{A,B}' & =\mu_{A\otimes B}\circ T(st_{A,B})\circ st'_{A,T(B)}.
\end{align*} 
The strong dagger Frobenius monad is \emph{commutative} when these coincide.
Theorem~\ref{thm:strong} restricts to an equivalence between commutative dagger Frobenius monoids and commutative strong dagger Frobenius monads.
Kleisli categories of commutative monads on symmetric monoidal categories are again symmetric monoidal~\cite{day:kleisli}. This extends to dagger categories.

\begin{theorem}
  If $T$ is a commutative strong dagger Frobenius monad on a symmetric monoidal dagger category $\cat{C}$, then $\Kl(T)$ is a symmetric monoidal dagger category.
\end{theorem}
\begin{proof} 
  The monoidal structure on $\Kl(T)$ is given by $A\otimes_T B=A\otimes B$ on objects and by $f\otimes_T g=\dst\circ (f\otimes g)$ on morphisms. 
  The coherence isomorphisms of $\Kl(T)$ are images of those in $\cat{C}$ under the functor $\cat{C}\to\Kl(T)$. 
  This functor preserves daggers and hence unitaries, making all coherence isomorphisms of $\Kl(T)$ unitary. 
  It remains to check that the dagger on $\Kl(T)$ satisfies $(f\otimes_T g)^\dag=f^\dag\otimes_T g^\dag$. Theorem~\ref{thm:strong} makes $T$ isomorphic to $S=-\otimes T(I)$, and this induces an isomorphism between the respective Kleisli categories that preserves daggers and monoidal structure on the nose. Thus it suffices to check that this equation holds on $\Kl(S)$:
  \[
    (f \otimes_S g)^\dag
    \quad = \quad
    \begin{pic}[scale=.5]
      \node[morphism,hflip] (f) at (0,0) {$f$};
      \node[morphism,hflip] (g) at (2,0) {$g$};
      \node[dot] (d) at (1.25,-.85) {};
      \node[dot] (e) at (2.5,-1.5) {};
      \draw (f.north) to +(0,1.25);
      \draw (g.north) to +(0,1.25);
      \draw ([xshift=-1mm]f.south west) to +(0,-2.5);
      \draw ([xshift=-1mm]g.south west) to +(0,-2.5);
      \draw ([xshift=1mm]f.south east) to[out=-90,in=180] (d.west);
      \draw ([xshift=1mm]g.south east) to[out=-90,in=0] (d.east);
      \draw (d.south) to[out=-90,in=180] (e.west);
      \draw (e.south) to +(0,-.5) node[dot]{};
      \draw (e.east) to[out=0,in=-90] +(1,1) to +(0,2.25) node[right] {$T(I)$};
    \end{pic}
    = \quad
    \begin{pic}[scale=.5]
      \node[morphism,hflip] (f) at (0,0) {$f$};
      \node[morphism,hflip] (g) at (2.5,0) {$g$};
      \node[dot] (d) at (2.75,1.2) {};
      \node[dot] (df) at (.75,-.85) {};
      \node[dot] (dg) at (3.25,-.85) {};
      \draw (f.north) to +(0,2);
      \draw ([xshift=-1mm]g.north) to[out=90,in=-90,looseness=.7] +(-1.5,2);
      \draw ([xshift=-1mm]f.south west) to +(0,-2);
      \draw ([xshift=-1mm]g.south west) to +(0,-2);
      \draw ([xshift=1mm]f.south east) to[out=-90,in=180] (df.west);
      \draw ([xshift=1mm]g.south east) to[out=-90,in=180] (dg.west);
      \draw (df.south) to +(0,-.5) node[dot]{};
      \draw (dg.south) to +(0,-.5) node[dot]{};
      \draw (df.east) to[out=0,in=180,looseness=.7] (d.west);
      \draw (dg.east) to[out=0,in=0] (d.east);
      \draw (d.north) to +(0,1) node[right] {$T(I)$};
    \end{pic}
    \quad = \quad
    f^\dag \otimes_S g^\dag\text{.}
  \]
  But this is a straightforward graphical argument
  
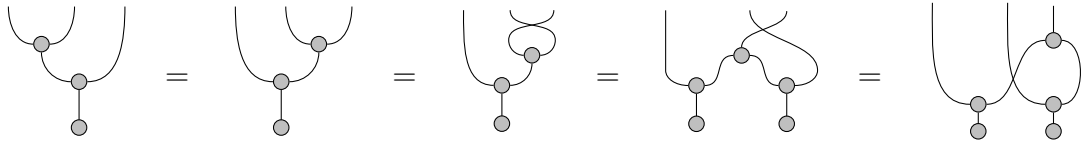
\[
    \begin{pic}[scale=.5]
      \node[dot] (l) at (0,0) {};
      \node[dot] (r) at (1,-1) {};
      \draw (r.south) to +(0,-1) node[dot]{};
      \draw (l.south) to[out=-90,in=180] (r.west);
      \draw (r.east) to[out=0,in=-90] +(1,2);
      \draw (l.east) to[out=0,in=-90] +(.66,1);
      \draw (l.west) to[out=180,in=-90] +(-.66,1);
    \end{pic}
    \quad = \quad
    \begin{pic}[scale=.5]
      \node[dot] (l) at (1,0) {};
      \node[dot] (r) at (0,-1) {};
      \draw (r.south) to +(0,-1) node[dot]{};
      \draw (l.south) to[out=-90,in=0] (r.east);
      \draw (r.west) to[out=180,in=-90] +(-1,2);
      \draw (l.west) to[out=180,in=-90] +(-.66,1);
      \draw (l.east) to[out=0,in=-90] +(.66,1);
    \end{pic}
    \quad = \quad    
    \begin{pic}[scale=.4]
      \node[dot] (l) at (1,0) {};
      \node[dot] (r) at (0,-1) {};
      \draw (r.south) to +(0,-1) node[dot]{};
      \draw (l.south) to[out=-90,in=0] (r.east);
      \draw (r.west) to[out=180,in=-90] +(-1,2.5);
      \draw (l.west) to[out=180,in=-90] +(-.5,.5) to[out=90,in=-90] +(1.5,1);
      \draw (l.east) to[out=0,in=-90] +(.5,.5) to[out=90,in=-90] +(-1.5,1);
    \end{pic}
    \quad = \quad
    \begin{pic}[yscale=.4,xscale=.6]
      \node[dot] (l) at (0,0) {};
      \node[dot] (m) at (1,1) {};
      \node[dot] (r) at (2,0) {};
      \draw (l.south) to +(0,-1) node[dot]{};
      \draw (r.south) to +(0,-1) node[dot]{};
      \draw (l.east) to[out=0,in=180] (m.west);
      \draw (m.east) to[out=0,in=180] (r.west);
      \draw (l.west) to[out=180,in=-90] +(-.5,.5) to +(0,2);
      \draw (m.north) to[out=90,in=-90] +(1,1.2);
      \draw (r.east) to[out=0,in=-90] +(.5,.5) to[out=90,in=-90] +(-1.5,2);
    \end{pic}
    \quad = \quad
    \begin{pic}[scale=.5]
      \node[dot] (l) at (0,0) {};
      \node[dot] (r) at (2,0) {};
      \node[dot] (t) at (2,1.7) {};
      \draw (l.south) to +(0,-.5) node[dot]{};
      \draw (r.south) to +(0,-.5) node[dot]{};
      \draw (l.east) to[out=0,in=180] (t.west);
      \draw (r.east) to[out=0,in=0] (t.east);
      \draw (t.north) to +(0,.7);
      \draw (l.west) to[out=180,in=-90] +(-1,2.7);
      \draw (r.west) to[out=180,in=-90] +(-1,2.7);
    \end{pic}
  \]
  using associativity, commutativity, the unit law, and Lemma~\ref{lem:extendedfrobeniuslaw}.
\end{proof}

\section{Closure}\label{sec:coherence}

This final section justifies the Frobenius law from first principles, by explaining it as a coherence property between daggers and closure. In a monoidal dagger category that is closed, monoids and daggers interact in two ways.
First, any monoid picks up an involution by internalizing the dagger.
Second, any monoid embeds into an endohomset by closure, and the dagger is an involution on the endohomset.
The Frobenius law is equivalent to the property that these two canonical involutions coincide.

We start by giving an equivalent formulation of the Frobenius law.

\begin{lemma}\label{lem:equivalentformofFrob}
  A monoid $(A,\tinymult,\tinyunit)$ in a monoidal dagger category is a dagger Frobenius monoid if and only if it satisfies the following equation.
  \begin{equation}\label{eq:equivalentfrobeniuslaw}
      \begin{pic}[scale=.4]
        \node[dot] (t) at (0,1) {};
        \draw (t) to +(0,1.5);
        \draw (t) to[out=0,in=90] (1,0) to (1,-1);
        \draw (t) to[out=180,in=90] (-1,0) to (-1,-1);
      \end{pic}
      \quad 
      =
      \quad
      \begin{pic}
        \node[dot] (a)  at (-.2,0.75) {};
        \node[dot] (d)  at (-.2,1.25) {}; 
        \node[dot] (b) at (.5,0) {};
        \draw (d) to (a);
        \draw (a) to[out=180,in=90] +(-.5,-1);
        \draw (b) to[out=0,in=-90] +(.35,.5) to +(0,.7);
        \draw (a) to[out=0,in=180]  (b);
        \draw (b) to +(0,-.45);
      \end{pic}
    \end{equation}
\end{lemma}
\begin{proof}
  The Frobenius law~\eqref{eq:frobeniuslaw} directly implies~\eqref{eq:equivalentfrobeniuslaw}.
  Conversely, \eqref{eq:equivalentfrobeniuslaw} gives
  \[
    \begin{pic}[yscale=0.75]
          \node (0a) at (-0.5,0) {};
          \node (0b) at (0.5,0) {};
          \node[dot] (1) at (0,1) {};
          \node[dot] (2) at (0,2) {};
          \node (3a) at (-0.5,3) {};
          \node (3b) at (0.5,3) {};
          \draw[out=90,in=180] (0a) to (1.west);
          \draw[out=90,in=0] (0b) to (1.east);
          \draw (1.north) to (2.south);
          \draw[out=180,in=270] (2.west) to (3a);
          \draw[out=0,in=270] (2.east) to (3b);
    \end{pic}
    \;=\;
    \begin{pic}
      \node[dot] (l) at (-.75,.5) {};
      \node[dot] (m) at (0,-.5) {};
      \node[dot] (r) at (1,.5) {};
      \draw (r.north) to (1,1) node[dot] {};
      \draw (r.east) to[out=0,in=90,looseness=.5] (1.5,-1);
      \draw (m.south) to (0,-1);
      \draw (m.east) to[out=0,in=180] (r.west);
      \draw (m.west) to[out=180,in=-90] (l.south);
      \draw (l.east) to[out=0,in=-90] (-.25,1);
      \draw (l.west) to[out=180,in=-90] (-1.25,1);
    \end{pic}
    \;=\;
    \begin{pic}[xscale=.75]
      \node[dot] (l) at (-.75,-.5) {};
      \node[dot] (m) at (0,0) {};
      \node[dot] (r) at (1,.5) {};
      \draw (l.south) to (-.75,-1);
      \draw (l.west) to[out=180,in=-90] (-1.5,1);
      \draw (l.east) to[out=0,in=-90] (m.south);
      \draw (m.west) to[out=180,in=-90] (-.5,1);
      \draw (m.east) to[out=0,in=180] (r.west);
      \draw (r.north) to (1,1) node[dot] {};
      \draw (r.east) to[out=0,in=90,looseness=.5] (1.5,-1);
    \end{pic}
    \;=\;
    \begin{pic}[yscale=0.75,xscale=-1]
          \node (0) at (0,0) {}; 
          \node (0a) at (0,1) {};
          \node[dot] (1) at (0.5,2) {};
          \node[dot] (2) at (1.5,1) {};
          \node (3) at (1.5,0) {};
          \node (4) at (2,3) {};
          \node (4a) at (2,2) {};
          \node (5) at (0.5,3) {};
          \draw (0) to (0a.center);
          \draw[out=90,in=180] (0a.center) to (1.east);
          \draw[out=0,in=180] (1.west) to (2.east);
          \draw[out=0,in=270] (2.west) to (4a.center);
          \draw (4a.center) to (4);
          \draw (2.south) to (3);
          \draw (1.north) to (5);
    \end{pic}
  \]
  by associativity. But since the left-hand side is self-adjoint, so is the right-hand side, giving the Frobenius law~\eqref{eq:frobeniuslaw}.
\end{proof}

Any dagger Frobenius monoid forms a duality with itself in the following sense.

\begin{definition}
  Morphisms $\eta \colon I \to A \otimes B$	and $\varepsilon \colon B \otimes A \to I$ in a monoidal category \emph{form a duality} when they satisfy the following equations.
  \[
    \begin{pic}
	  \node[morphism] (u) at (0,0) {$\eta$};
	  \node[morphism] (c) at (.75,.75) {$\varepsilon$};
	  \draw ([xshift=1mm]u.north east) to[out=90,in=-90] ([xshift=-1mm]c.south west);
	  \draw ([xshift=-1mm]u.north west) to +(0,1);
	  \draw ([xshift=1mm]c.south east) to +(0,-1);
    \end{pic}
    \; = \; 
    \begin{pic}
      \draw (0,0) to (0,1.7);
    \end{pic}
    \qquad\qquad
    \begin{pic}[xscale=-1]
	  \node[morphism] (u) at (0,0) {$\eta$};
	  \node[morphism] (c) at (.75,.75) {$\varepsilon$};
	  \draw ([xshift=1mm]u.north west) to[out=90,in=-90] ([xshift=-1mm]c.south east);
	  \draw ([xshift=-1mm]u.north east) to +(0,1);
	  \draw ([xshift=1mm]c.south west) to +(0,-1);
    \end{pic}
    \; = \; 
    \begin{pic}
      \draw (0,0) to (0,1.7);
    \end{pic}
  \]
\end{definition}

In the categories $\cat{Rel}$ and $\cat{FHilb}$, every object $A$ is part of a duality $(A,B,\eta,\varepsilon)$: they are \emph{compact} categories. Moreover, in those categories we may choose $\varepsilon = \eta^\dag \circ \sigma$: they are \emph{compact dagger categories}.

Let \cat{C} be a closed monoidal category, so that there is a correspondence of morphisms $A \otimes B \to C$  and $A \to [B, C]$ called \emph{currying}. Write $A^*$ for $[A,I]$, and write $\ev$ for the counit $A^* \otimes A \to I$. 
If $(A,\tinycomult,\tinycounit)$ is a comonoid in $\cat{C}$, then $A^*$ becomes a monoid with unit and multiplication given by currying $\tinycounit \colon I \otimes A \to I$ and 
\[
  	\begin{pic}[yscale=.7]
   		\node[morphism] (a)  at (0,0.75) {ev}; 
   		\node[morphism] (c)  at (-.25,2) {ev}; 
    	\node[dot] (b) at (.5,0) {};
   		\draw ([xshift=-1mm]a.south west) to +(0,-1);
   		\draw ([xshift=-1mm]c.south west) to[out=-90,in=90] +(-.25,-2.25);
   		\draw (b.east) to[out=0,in=-90] +(.25,.75) to[out=90,in=-90] (c.south east);
   		\draw (a.south east) to[out=-90,in=180]  (b.west);
   		\draw (b) to +(0,-.5);
  	\end{pic} \;\colon (A^* \otimes A^*) \otimes A \to I \text{.}
\]
This monoid and comonoid are related by a map $i\colon A\to A^*$ obtained by currying 
$\begin{pic}[scale=.3]
      \node[dot] (b) at (1,0) {};
      \node[dot] (c) at (1,1) {};
      \draw (b) to  (c);
      \draw (b) to[out=180,in=90] (0,-1);
      \draw (b) to[out=0,in=90] (2,-1);
\end{pic}$.

\begin{proposition}\label{prop:closedfrobenius}
  A monoid $(A,\tinymult,\tinyunit)$ in a monoidal dagger category that is also a closed monoidal category is a dagger Frobenius monoid if and only if $i\colon A\to A^*$ is a monoid homomorphism and $\ev \colon A^* \otimes A \to I$ forms a duality with
  \begin{equation}\label{eq:frobeniusduality}
    	\begin{pic}[scale=.4]
	      	\node[dot] (a) at (0,0) {};
	      	\node[dot] (b) at (0,-1) {};
	      	\node[morphism] (c) at (1,1.5) {$i$};
	      	\draw (b.north) to  (a.south);
	      	\draw (a.east) to[out=0, in=-90]  (c.south);
	      	\draw (a.west) to[out=180,in=-90,looseness=.8] (-1,3);
	      	\draw ([yshift=.5mm]c.north) to  (1,3);
    	\end{pic}\;\colon I \to A \otimes A^* \text{.}
  \end{equation}
\end{proposition}
\begin{proof}
  The morphism $i$ always preserves units: $\ev \circ (i \otimes \id) \circ (\tinyunit \otimes \id) = \tinycounit \circ \tinymult \circ (\tinyunit \otimes \id) = \tinycounit$.
	It preserves multiplication precisely when:
	\begin{equation}\label{eq:altfrob1}
      \begin{pic}[scale=.4]
        \node[dot] (t) at (0,1) {};
        \node[dot] (b) at (1,0) {};
        \node[dot] (c) at (0,2) {};
        \draw (c) to (t);
        \draw (t) to[out=0,in=90] (b);
        \draw (t) to[out=180,in=90] (-1,0) to (-1,-1);
        \draw (b) to[out=180,in=90] (0,-1);
        \draw (b) to[out=0,in=90] (2,-1);
      \end{pic}
      \quad = \quad
      \begin{pic}[yscale=.4,xscale=-.4]
        \node[dot] (t) at (0,1) {};
        \node[dot] (b) at (1,0) {};
        \node[dot] (c) at (0,2) {};
        \draw (c) to (t);
        \draw (t) to[out=0,in=90] (b);
        \draw (t) to[out=180,in=90] (-1,0) to (-1,-1);
        \draw (b) to[out=180,in=90] (0,-1);
        \draw (b) to[out=0,in=90] (2,-1);
      \end{pic}
      \quad = \quad
      \begin{pic}[scale=.4]
      	\node[dot] (b) at (1.25,-.25) {};
      	\node[morphism] (c) at (1.25,1) {$i$};
      	\node[morphism] (d) at (2.5,3.2) {ev};
      	\draw (b.north) to  (c.south);
      	\draw (b.west) to[out=180,in=90] +(-.75,-1);
      	\draw (b.east) to[out=0,in=90] +(.75,-1);
      	\draw ([xshift=-1.5mm]d.south west) to[out=-90,in=90] (c.north);
      	\draw ([xshift=1mm]d.south east) to +(0,-4);
   	  \end{pic}
   	  \quad = \quad
      \begin{pic}[scale=.4]
      	\node[dot] (a) at (0,1) {};
      	\node[morphism] (b) at (-.8,-.5) {$i$};
      	\node[morphism] (c) at (.8,-.5) {$i$};
      	\node[morphism] (d) at (1.5,2.75) {ev};
      	\draw (a.west) to[out=180,in=90] (b.north);
      	\draw (a.east) to[out=0,in=90] (c.north);
      	\draw (a.north) to[out=90,in=-90] ([xshift=-1mm]d.south west);
      	\draw (b.south) to +(0,-.75);
     	\draw (c.south) to +(0,-.75);
      	\draw ([xshift=1mm]d.south east) to[out=-90,in=90] +(0,-4);
      \end{pic}
      \quad = \quad
      \begin{pic}[scale=.75]
   		\node[dot] (a)  at (-.5,0.75) {};
   		\node[dot] (d)  at (-.5,1.25) {}; 
   		\node[dot] (c)  at (-.5,1.75) {}; 
   		\node[dot] (e)  at (-.5,2.25) {};
    	\node[dot] (b) at (.2,.2) {};
    	\draw (d) to (a);
    	\draw (e) to (c);
   		\draw (a.west) to[out=180,in=90,looseness=.8] +(-.25,-1);
   		\draw (c.west) to[out=180,in=90,looseness=.8] +(-.75,-2);
   		\draw (b.east) to[out=0,in=0] (c.east);
   		\draw (a) to[out=0,in=180]  (b);
   		\draw (b) to +(0,-.45);
      \end{pic}
	\end{equation}
	Furthermore, $\ev \colon A^* \otimes A \to I$ and~\eqref{eq:frobeniusduality} form a duality precisely when:
	\begin{equation}\label{eq:altfrob2}
  	  \begin{pic}
	  	\draw (0,0) to (0,2);
	  \end{pic}
	  \; = \;
      \begin{pic}[scale=.4]
      	\node[dot] (a) at (0,0) {};
      	\node[dot] (b) at (0,-1) {};
      	\node[morphism] (c) at (1,1.5) {$i$};
      	\node[morphism] (d) at (1.6,3) {ev};
      	\draw (b) to  (a);
      	\draw (a.east) to[out=0, in=-90]  (c.south);
      	\draw (a.west) to[out=180,in=-90,looseness=.6] (-1,3.5);
      	\draw ([xshift=-2mm]d.south west) to  (c.north);
      	\draw ([xshift=2mm]d.south east) to  +(0,-4);
      \end{pic}
	  \; = \;
   	  \begin{pic}[scale=.4]
      	\node[dot] (a) at (0,0) {};
      	\node[dot] (b) at (0,-1) {};
      	\node[dot] (c) at (1.75,1.5) {};
      	\node[dot] (d) at (1.75,2.5) {};
      	\draw (b) to  (a);
      	\draw (d) to  (c);
      	\draw (a) to[out=0, in=180]  (c);
      	\draw (a.west) to[out=180,in=-90,looseness=.6] (-1,3.5);
      	\draw (c.east) to[out=0, in=90,looseness=.8]  +(1,-3);
      \end{pic}
      \qquad \qquad \qquad
  	  \begin{pic}
 	 		\draw (0,0) to (0,2);
	  \end{pic}
  	  \; = \;
      \begin{pic}[scale=.4]
      	\node[dot] (a) at (.25,0) {};
      	\node[dot] (b) at (.25,-1) {};
      	\node[morphism] (c) at (1,1.5) {$i$};
      	\node[morphism] (d) at (-1,1.5) {ev};
      	\draw (b) to  (a);
      	\draw (a) to[out=0, in=-90]  (c.south);
      	\draw ([xshift=1mm]d.south east) to[out=-90,in=180] (a);
      	\draw ([xshift=-1mm]d.south west) to +(0,-2.5);
      	\draw ([yshift=.5mm]c.north) to  (1,3);
      \end{pic}
	\end{equation}
	By evaluating both sides, it is easy to see that the left equation implies the right one.
  
    Now, assuming the Frobenius law~\eqref{eq:frobeniuslaw}, Lemma~\ref{lem:extendedfrobeniuslaw} and the unit law guarantee that (the left equation of)~\eqref{eq:altfrob2} is satisfied, as well as~\eqref{eq:altfrob1}:
    \[  
      \begin{pic}[scale=.75]
   		\node[dot] (a)  at (-.5,0.75) {};
   		\node[dot] (d)  at (-.5,1.25) {}; 
   		\node[dot] (c)  at (-.5,1.75) {}; 
   		\node[dot] (e)  at (-.5,2.25) {};
    	\node[dot] (b) at (.2,.2) {};
    	\draw (d) to (a);
    	\draw (e) to (c);
   		\draw (a.west) to[out=180,in=90,looseness=.8] +(-.25,-1);
   		\draw (c.west) to[out=180,in=90,looseness=.8] +(-.75,-2);
   		\draw (b.east) to[out=0,in=0] (c.east);
   		\draw (a) to[out=0,in=180]  (b);
   		\draw (b) to +(0,-.45);
      \end{pic}
      \quad = \quad
      \begin{pic}[scale=.75]
		\node[dot] (b) at (0,3) {};
		\node[dot] (c) at (0,2) {};
		\node[dot] (d) at (0,1.5) {};
		\draw (d.east) to[out=0,in=90] +(.3,-.5);
		\draw (d.west) to[out=180,in=90] +(-.3,-.5);
		\draw (d.north) to (c.south);
		\draw (c.west) to[out=180,in=-90] +(-.25,.5) node[dot]{};
		\draw (c.east) to[out=0,in=0] (b.east);
		\draw (b.north) to +(0,.5) node[dot]{};
		\draw (b.west) to[out=180,in=90] +(-1,-2);
      \end{pic}
      \quad = \quad
      \begin{pic}[scale=.4]
        \node[dot] (t) at (0,1) {};
        \node[dot] (b) at (1,0) {};
        \node[dot] (c) at (0,2) {};
        \draw (c) to (t);
        \draw (t) to[out=0,in=90] (b);
        \draw (t) to[out=180,in=90] (-1,0) to (-1,-1);
        \draw (b) to[out=180,in=90] (0,-1);
        \draw (b) to[out=0,in=90] (2,-1);
      \end{pic}
    \]
    Conversely, equations~\eqref{eq:altfrob1} and~\eqref{eq:altfrob2} imply:
    \[
      \begin{pic}[scale=.9]
        \node[dot] (a)  at (-.2,0.75) {};
        \node[dot] (d)  at (-.2,1.15) {}; 
        \node[dot] (b) at (.5,0) {};
        \draw (d) to (a);
        \draw (a) to[out=180,in=90,looseness=.8] +(-.5,-1.25);
        \draw (b) to[out=0,in=-90,looseness=.8] +(.5,1.5);
        \draw (a) to[out=0,in=180]  (b);
        \draw (b) to +(0,-.45);
      \end{pic}
      \quad = \quad
      \begin{pic}[scale=.5]
	    \node[dot] (a) at (0,2.2) {};
	    \node[dot] (b) at (1.5,3) {};
	    \node[dot] (c) at (2,.2) {};
	    \node[dot] (d) at (1,1) {};
	    \draw (a.south) to +(0,-.5) node[dot]{};
	    \draw (b.north) to +(0,.5) node[dot]{};
	    \draw (d.north) to +(0,.5) node[dot]{};
	    \draw (a.east) to[out=0,in=180] (b.west);
	    \draw (b.east) to[out=0,in=0,looseness=.8] (c.east);
	    \draw (c.west) to[out=180,in=0] (d.east);
	    \draw (c.south) to +(0,-.5);
	    \draw (d.west) to[out=180,in=90,looseness=.7] +(-.5,-1.5);
	    \draw (a.west) to[out=180,in=-90,looseness=.7] +(-.5,2);
      \end{pic}
      \quad = \quad
      \begin{pic}[scale=.5]
      	\node[dot] (l) at (0,0) {};
      	\node[dot] (m) at (1.25,1) {};
      	\node[dot] (r) at (2,0) {};
      	\draw (l.south) to +(0,-.5) node[dot]{};
      	\draw (m.north) to +(0,.5) node[dot]{};
      	\draw (l.west) to[out=180,in=-90,looseness=.8] +(-.5,2);
      	\draw (l.east) to[out=0, in=180] (m.west);
      	\draw (m.east) to[out=0,in=90] (r.north);
      	\draw (r.west) to[out=180,in=90] +(-.5,-1.5);
      	\draw (r.east) to[out=0,in=90] +(.5,-1.5);
      \end{pic}
      \quad = \quad
      \begin{pic}[scale=.5]
        \node[dot] (t) at (0,1) {};
        \draw (t) to +(0,1.5);
        \draw (t) to[out=0,in=90] (1,0) to (1,-1);
        \draw (t) to[out=180,in=90] (-1,0) to (-1,-1);
      \end{pic}
    \]
    Lemma~\ref{lem:equivalentformofFrob} now finishes the proof.
\end{proof}

In any closed monoidal category, $[A,A]$ is canonically a monoid. The `way of the dagger' suggests that there should be interaction between the dagger and closure in categories that have both.

\begin{definition}
  A \emph{sheathed dagger category} is a monoidal dagger category that is also closed monoidal, such that
  \[
    \begin{pic}[scale=.3]
      \node[morphism, hflip] (a) at (3.5,-2) {$\ev_{[A,A]}$};
      \node[dot] (b) at (1,-.25) {};
      \node[dot] (c) at (1,1) {};
      \draw (b) to  (c);
      \draw (b) to[out=180,in=90] (0,-1) to (0,-3.5);
      \draw (b) to[out=0,in=90] (a.north west);
      \draw (a.south) to +(0,-1);
      \draw (a.north east) to +(0,3);
    \end{pic}
    \quad = \quad
    \ev_{[A,A]} 
    \;\colon [A,A] \otimes A \to A
  \]
  for all objects $A$, and for all morphisms $f,g \colon B \to C \otimes [A,A] \colon$
  \[
    \begin{pic}[yscale=.6]
      \node[morphism] (a) at (0,0) {$f$};
      \node[morphism] (b) at (0.63,1) {$\ev_{[A,A]}$};
      \draw (a.south) to +(0,-.5);
      \draw ([xshift=-1mm]a.north west) to +(0,1.5);
      \draw ([xshift=1mm]a.north east) to (b.south west);
      \draw ([xshift=-1mm]b.north) to +(0,.5);
      \draw ([xshift=-1mm]b.south east) to +(0,-1.5);
    \end{pic}
    \; = \;
    \begin{pic}[yscale=.6]
      \node[morphism] (a) at (0,0) {$g$};
      \node[morphism] (b) at (0.62,1) {$\ev_{[A,A]}$};
      \draw (a.south) to +(0,-.5);
      \draw ([xshift=-1mm]a.north west) to +(0,1.5);
      \draw ([xshift=1mm]a.north east) to (b.south west);
      \draw ([xshift=-1mm]b.north) to +(0,.5);
      \draw ([xshift=-1mm]b.south east) to +(0,-1.5);
    \end{pic}
    \qquad \implies \qquad 
    \begin{pic}
      \node[morphism] (a) at (0,0) {$f$};
      \draw (a.south) to +(0,-.5);
      \draw ([xshift=-1mm]a.north west) to +(0,.5);
      \draw ([xshift=1mm]a.north east) to +(0,.5);
    \end{pic}
    \; = \;
    \begin{pic}
      \node[morphism] (a) at (0,0) {$g$};
      \draw (a.south) to +(0,-.5);
      \draw ([xshift=-1mm]a.north west) to +(0,.5);
      \draw ([xshift=1mm]a.north east) to +(0,.5);
    \end{pic}
  \]
\end{definition}

Any compact dagger category is a sheathed dagger category: the first axiom there says that $A^* \otimes A$ with its canonical monoid structure is a dagger Frobenius monoid, and the second axiom then holds because the evaluation morphism is invertible. 
In principle, the definition of sheathed dagger categories is much weaker. Although we have no uncontrived examples of sheathed dagger categories that are not compact dagger categories, we will work with the more general sheathed dagger categories because they are the natural home for the following arguments.
The second axiom merely says that partial evaluation is faithful, which is the case in any well-pointed monoidal dagger category.
In any closed monoidal category, the evaluation map canonically makes $A$ into an algebra for the monad $-\otimes [A,A]$. The first axiom merely says that $A$ is $-\otimes [A,A]$-self-adjoint, as in \eqref{eq:self-adjoint}. It does not assume $[A,A]$ is a dagger Frobenius monoid, nor that $A$ is a FEM-algebra. 
No other plausible conditions are imposed, such as the bifunctor $[-,-]$ being a dagger functor, which does hold in compact dagger categories.
Nevertheless, the following example shows that being a sheathed dagger category is an essentially monoidal notion that degenerates for cartesian categories.

\begin{example}
  If a Cartesian closed category has a dagger, every homset is a singleton. 
\end{example}
\begin{proof}
  Because the terminal object is in fact a zero object, there are natural bijections
  $\hom (A,B)\cong \hom(0\times A,B)\cong \hom (0,B^A)\cong \{*\}$.
\end{proof}

Currying the multiplication of a monoid $(A,\tinymult,\tinyunit)$ in a closed monoidal category gives a monoid homomorphism $R \colon A \to [A,A]$. This is the abstract version of Cayley's embedding theorem, which states that any group embeds into the symmetric group on itself.
If the category also has a dagger, there is also a monoid homomorphism $R_* = [R^\dag,\id[I]] \colon A^* \to [A,A]^*$.

\begin{theorem}\label{thm:sheathedfrobenius}
  In a sheathed dagger category, $(A,\tinymult,\tinyunit)$ is a dagger Frobenius monoid if and only if the following diagram commutes.
  \[
    \begin{tikzpicture}
     \matrix (m) [matrix of math nodes,row sep=2em,column sep=4em,minimum width=2em]
     {
      A & {[A,A]} \\
      A^* & {[A,A]^*} \\};
     \path[->]
     (m-1-1) edge node [left] {$i_A$} (m-2-1)
             edge node [above] {$R$} (m-1-2)
     (m-2-1) edge node [below] {$R_*$} (m-2-2)
     (m-1-2) edge node [right] {$i_{[A,A]}$} (m-2-2);
    \end{tikzpicture}
    \]
\end{theorem}
\begin{proof}
  Evaluating both sides shows that $R$ and $i$ commute precisely when:
  \[
    \begin{pic}[scale=.4]
      \node[dot] (b) at (1,0) {};
      \node[dot] (c) at (1,1) {};
      \node[morphism] (a) at (-.25,-1.5) {$R$};
      \draw (b) to  (c);
      \draw (b) to[out=180,in=90] (a.north);
      \draw (a.south) to +(0,-1);
      \draw (b) to[out=0,in=90] (2,-1) to (2,-3);
    \end{pic}
    \quad 
    = 
    \quad 
    \begin{pic}[scale=.4]
      \node[dot] (b) at (1,0) {};
      \node[dot] (c) at (1,1) {};
      \node[morphism, hflip] (a) at (1.75,-1.5) {$R$};
      \draw (b) to  (c);
      \draw (a.south) to +(0,-1);
      \draw (b) to[out=180,in=90] (0,-1) to (0,-3);
      \draw (b) to[out=0,in=90] (a.north);
    \end{pic}
  \]
  But this is equivalent to
  \[
    \begin{pic}[scale=.4]
        \node[dot] (t) at (0,1) {};
        \draw (t) to +(0,1.5);
        \draw (t) to[out=0,in=90] (1,0) to (1,-1);
        \draw (t) to[out=180,in=90] (-1,0) to (-1,-1);
    \end{pic}
    \quad = \quad
    \begin{pic}
      \node[morphism] (a) at (0,0) {$R$};
      \node[morphism] (b) at (0.44,1) {$\ev_{[A,A]}$};
      \draw (a.south) to +(0,-.5);
      \draw (a.north) to (b.south west);
      \draw ([xshift=-1mm]b.north) to +(0,.5);
      \draw ([xshift=-1mm]b.south east) to +(0,-1.5);
    \end{pic}
    \quad = \quad
    \begin{pic}[scale=.4]
      \node[dot] (b) at (1,0) {};
      \node[dot] (c) at (1,1) {};
      \node[morphism] (a) at (-.25,-1.5) {$R$};
      \node[morphism, hflip] (d) at (2.75,-3.5) {$\ev_{[A,A]}$};
      \draw (b) to  (c);
      \draw (b) to[out=180,in=90] (a.north);
      \draw (a.south) to +(0,-3);
      \draw (b) to[out=0,in=90] (2,-1) to (2,-3);
      \draw (d.north east) to +(0,5);
      \draw (d.south) to +(0,-1);
    \end{pic}
    \quad 
    = 
    \quad 
    \begin{pic}[scale=.4]
      \node[dot] (b) at (1,0) {};
      \node[dot] (c) at (1,1) {};
      \node[morphism, hflip] (a) at (1.75,-1.5) {$R$};
      \node[morphism, hflip] (d) at (2.75,-3.5) {$\ev_{[A,A]}$};
      \draw (b) to  (c);
      \draw (a.south) to +(0,-1);
      \draw (d.north east) to +(0,5);
      \draw (d.south) to +(0,-1);
      \draw (b) to[out=180,in=90] (0,-1) to (0,-5);
      \draw (b) to[out=0,in=90] (a.north);
    \end{pic}
    \quad = \quad 
    \begin{pic}[scale=.4]
      \node[dot] (b) at (1,0) {};
      \node[dot] (c) at (1,1) {};
      \node[dot] (d) at (2.75,-2) {};
      \draw (b) to  (c);
      \draw (d) to[out=0, in=-90] +(1,1) to +(0,2);
      \draw (d.south) to +(0,-1);
      \draw (b) to[out=180,in=90] (0,-1) to (0,-3.5);
      \draw (b) to[out=0,in=180] (d);
    \end{pic} 
  \]
  Lemma~\ref{eq:equivalentfrobeniuslaw} now finishes the proof.
\end{proof}

\begin{corollary} 
  The following are equivalent for a monoid $(A,\tinymult,\tinyunit)$ in a compact dagger category:
  \begin{itemize}
  	\item $(A,\tinymult,\tinyunit)$ is a dagger Frobenius monoid;
  	\item the canonical morphism $i \colon A \to A^*$ is an involution: $i_* \circ i = \id[A]$;
  	\item the canonical Cayley embedding is involutive: $i \circ R = R_* \circ i$.
  \end{itemize}
\end{corollary}
\begin{proof}
  Combine Proposition~\ref{prop:closedfrobenius} and Theorem~\ref{thm:sheathedfrobenius}.
\end{proof}

\bibliographystyle{plain}

\begin{thebibliography}{10}

\bibitem{axelsenkaarsgaard:reversiblerecursion}
H.~B. Axelsen and R.~Kaarsgaard.
\newblock Join inverse categories as models of reversible recursion.
\newblock In {\em Foundations of Software Science and Computation Structures},
  2016.

\bibitem{bowmanetal:traced}
W.~J. Bowman, R.~P. James, and A.~Sabry.
\newblock Dagger traced symmetric monoidal categories and reversible
  programming.
\newblock {\em Reversible Computation}, 2011.

\bibitem{bulacutorrecillas:extensions}
D.~Bulacu and B.~Torrecillas.
\newblock On {F}robenius and separable algebra extensions in monoidal
  categories: applications to wreaths.
\newblock {\em Journal of Noncommutative Geometry}, 9:707--774, 2015.

\bibitem{coeckepavlovic:measurement}
B.~Coecke and D.~Pavlovic.
\newblock Quantum measurements without sums.
\newblock In G.~Chen, L.~Kauffman, and S.~J. Lomonaco, editors, {\em The
  mathematics of quantum computation and technology}, pages 559--596. Taylor
  and Francis, 2008.

\bibitem{day:kleisli}
B.~Day.
\newblock On closed category of functors {II}.
\newblock In {\em Sydney Category Theory Seminar}, number 420 in Lecture Notes
  in Mathematics, pages 20--54, 1974.

\bibitem{devosdebaerdemacker:matrix}
A.~{de Vos} and S.~{de Baerdemacker}.
\newblock Matrix calculus for classical and quantum circuits.
\newblock {\em ACM Journal on Emerging Technologies in Computing Systems},
  11(2):9, 2014.

\bibitem{greenetal:quipper}
A.~S. Green, P.~{LeFanu} Lumsdaine, N.~J. Ross, P.~Selinger, and B.~Valiron.
\newblock A scalable quantum programming language.
\newblock {\em ACM SIGPLAN Notices}, 48(6):333--342, 2013.

\bibitem{heinosaariziman:quantum}
T.~Heinosaari and M.~Ziman.
\newblock {\em The mathematical language of quantum theory}.
\newblock Cambridge University Press, 2012.

\bibitem{heunen:thesis}
C.~Heunen.
\newblock {\em Categorical quantum models and logics}.
\newblock Amsterdam University Press, 2009.

\bibitem{heunen:embedding}
C.~Heunen.
\newblock An embedding theorem for {H}ilbert categories.
\newblock {\em Theory and Applications of Categories}, 22(13):321--344, 2009.

\bibitem{heunencontrerascattaneo:groupoids}
C.~Heunen, I.~Contreras, and A.~Cattaneo.
\newblock Relative {F}robenius algebras are groupoids.
\newblock {\em Journal of Pure and Applied Algebra}, 217:114--124, 2013.

\bibitem{heunenjacobs:daggerkernels}
C.~Heunen and B.~Jacobs.
\newblock Quantum logic in dagger kernel categories.
\newblock {\em Order}, 27(2):177--212, 2010.

\bibitem{heunenkarvonen:reversible}
C.~Heunen and M.~Karvonen.
\newblock Reversible monadic computing.
\newblock In {\em Mathematical Foundations of Programming Semantics}, volume
  319 of {\em Electronic Notes in Theoretical Computer Science}, pages
  217--237, 2015.

\bibitem{heunentull:groupoids}
C.~Heunen and S.~Tull.
\newblock Categories of relations as models of quantum theory.
\newblock In {\em International Workshop on Quantum Physics and Logic}, volume
  195 of {\em Electronic Proceedings in Theoretical Computer Science}, pages
  247--261, 2015.

\bibitem{heunenvicary:cqm}
C.~Heunen and J.~Vicary.
\newblock {\em Categories for Quantum Theory: An Introduction}.
\newblock Oxford University Press, 2016.

\bibitem{jacobs:weakening}
B.~P.~F. Jacobs.
\newblock Semantics of weakening and contraction.
\newblock {\em Annals of Pure and Applied Logic}, 69(1):73--106, 1994.

\bibitem{jacobs:walks}
B.~P.~F. Jacobs.
\newblock Coalgebraic walks, in quantum and {T}uring computation.
\newblock {\em FoSSaCS, Lecture Notes in Computer Science}, 6604:12--26, 2011.

\bibitem{jacobs:involutive}
B.P.F. Jacobs.
\newblock Involutive categories and monoids, with a {GNS}-correspondence.
\newblock {\em Foundations of Physics}, 42(7):874--895, 2012.

\bibitem{kock:strong}
A.~Kock.
\newblock Strong functors and monoidal monads.
\newblock {\em Archiv der Mathematik}, 23(1):113--120, 1972.

\bibitem{lauda:ambidextrous}
A.~Lauda.
\newblock Frobenius algebras and ambidextrous adjunctions.
\newblock {\em Theory and Applicatoins of Categories}, 16(4):84--122, 2006.

\bibitem{maclane:categories}
S.~{Mac Lane}.
\newblock {\em Categories for the Working Mathematician}.
\newblock Springer, 2nd edition, 1971.

\bibitem{pavlovic:abstraction}
D.~Pavlovic.
\newblock Geometry of abstraction in quantum computation.
\newblock In {\em Classical and Quantum Information Assurance Foundations and
  Practice}, number 09311 in Dagstuhl Seminar Proceedings, 2010.

\bibitem{selinger:cpm}
P.~Selinger.
\newblock Dagger compact closed categories and completely positive maps.
\newblock In {\em 3rd International Workshop on Quantum Programming Languages},
  volume 170 of {\em Electronic Notes in Theoretical Computer Science}, pages
  139--163, 2007.

\bibitem{selinger:graphicallanguages}
P.~Selinger.
\newblock A survey of graphical languages for monoidal categories.
\newblock Number 813 in Lecture Notes in Physics, pages 289--356. Springer,
  2009.

\bibitem{selingervaliron:lambda}
P.~Selinger and B.~Valiron.
\newblock A lambda calculus for quantum computation with classical control.
\newblock {\em Mathematical Structures in Computer Science}, 16(3):527--552,
  2006.

\bibitem{street1972formal}
R.~Street.
\newblock The formal theory of monads.
\newblock {\em Journal of Pure and Applied Algebra}, 2(2):149--168, 1972.

\bibitem{street:ambidextrous}
R.~Street.
\newblock Frobenius monads and pseudomonoids.
\newblock {\em Journal of Mathematical Physics}, 45(10):3930--3948, 2004.

\bibitem{hott}
{Univalent Foundations Program}.
\newblock {\em Homotopy Type Theory: Univalent Foundations of Mathematics}.
\newblock \url{http://homotopytypetheory.org/book}, Institute for Advanced
  Study, 2013.

\bibitem{vicary:quantumalgebras}
J.~Vicary.
\newblock Categorical formulation of quantum algebras.
\newblock {\em Communications in Mathematical Physics}, 304(3):765--796, 2011.

\bibitem{vicary:daggerlimits}
J.~Vicary.
\newblock Completeness of $\dag$-categories and the complex numbers.
\newblock {\em Journal of Mathematical Physics}, 52(8):082104, 2011.

\bibitem{watts:eilenberg}
C.~E. Watts.
\newblock Intrinsic charcterizations of some additive functors.
\newblock {\em Proceedings of the American Mathematical Society}, 11(1):5--8,
  1960.

\bibitem{wolff:monads}
H.~Wolff.
\newblock Monads and monoids on symmetric monoidal closed categories.
\newblock {\em Archiv der Mathematik}, 24(1):113--120, 1973.

\end{thebibliography}

\end{document}